\documentclass{amsart}

\usepackage{amssymb,amscd}
 \usepackage{subfigure}
 \usepackage[dvips]{graphics}
\font\nt=cmr7
\def\note#1
{\marginpar
{\nt $\leftarrow$
\par
\hfuzz=20pt \hbadness=9000 \hyphenpenalty=-100 \exhyphenpenalty=-100
\pretolerance=-1 \tolerance=9999 \doublehyphendemerits=-100000
\finalhyphendemerits=-100000 \baselineskip=6pt
#1}\hfuzz=1pt}

\newtheorem{thm}{Theorem}
\newtheorem{prop}{Proposition} 
\newtheorem{rem}{Remark}

\newtheorem{defin}{Definition}

\theoremstyle{remark}

\newcommand{\x}{\times}

\newcommand{\C}{\mathbb C}

\newcommand{\D}{\Delta}
\newcommand{\e}{\varepsilon}

\newcommand{\R}{\mathbb R}

\newcommand{\U}{\mathcal U}
\newcommand{\V}{U_h(gl_2)}

\newcommand{\ue}{{\mathcal U}_{\varepsilon}}
\newcommand{\p}{\partial}
\newcommand{\Dt}{D_t}

\newcommand{\z}{\mathbb Z}
\newcommand{\pl}{{\mathbb R}^2}
\newcommand{\cc}{\mathcal C}
\newcommand{\rc}{\mathcal R}
\newcommand{\dc}{\mathcal D}
\newcommand{\oti}{\otimes}
\newcommand{\s}{\sigma}

\newcommand{\Ob}{\mathop{\mathrm{Ob}}\nolimits}
\newcommand{\id}{\mathop{\mathrm{id}}\nolimits}
\newcommand{\Hom}{\mathop{\mathrm{Hom}}\nolimits}

\begin{document}

\setlength{\baselineskip}{16pt}


\title[Invariants of tangles with flat connections]{Invariants of
tangles with flat connections in their complements}
\author{R. Kashaev}
\address{Section de math\'ematiques, Universit\'e de Gen\`eve
CP 240, 2-4 rue du Li\`evre, CH 1211 Gen\`eve 24, Suisse}
\email{Rinat.Kashaev@math.unige.ch}
\author{N. Reshetikhin}
\address{Department of mathematics, University of California,
Berkeley, CA 94720, USA}
\email{reshetik@math.berkeley.edu}
 \copyrightinfo{2000}
    {American Mathematical Society}
\date{July, 2002}


\begin{abstract} Let $G$ be a simple complex algebraic group.
By using a notion of a $G$-category we define invariants of tangles with flat
$G$-connections in their complements. We also show that quantized
universal enveloping algebras at roots of unity provide examples
of $G$-categories.
\end{abstract}
\maketitle
\tableofcontents
\section*{Introduction}

A breakthrough in the theory of invariants of knots came with the
discovery of the Jones polynomial \cite{J} and its generalizations
(HOMFLY and Kauffman). Then it was shown in \cite{RT} that such
invariants as well as invariants of tangled graphs can be obtained from
quantized universal enveloping algebras of simple Lie
algebras.

In this paper we extend the construction of invariants of
links from \cite{RT} (for details see \cite{T-1}). We construct
invariants of tangles in $\pl\times[0,1]$
 with a flat connection in a principal $G$-bundle (where
$G$ is a simple complex algebraic group) over the complement of a
tangle. The invariant is a functor from the category of $G$-tangles
to a given $G$-category (see sections~\ref{s1}--\ref{s31} for
definitions).
The invariant of a tangle with a gauge class of flat connections in the
complement is defined as the image of the corresponding $G$-tangle.

In section~\ref{s1}  of this paper first we describe the moduli space of
flat connections in the complement to a tangle in ${\mathbb
R}^2\times [0,1]$. Then we define the category of tangles with
gauge classes of flat connections in their complements. The
construction is essentially the same as $\pi$-tangles introduced
and studied by Turaev in \cite{T-2}. In section~\ref{s3} we define the
category of
$G$-colored tangle diagrams and show that it is naturally equivalent to
the category of tangles with flat $G$-connections in their
complements.
In section~\ref{s31} we
introduce the notion of a $G$-category, which is very similar but
different
from that of \cite{T-2}. In section~\ref{sec4}
we construct invariants of $G$-tangles. In
sections~\ref{sec5}--\ref{sec7}
we show
that representations of $\ue$ provide examples of
$GL_2^*$-categories and therefore provide invariants of tangles
with $GL_2$-connections in their complements.

We are grateful to R. Kirby, D. Thurston, V. Turaev, and M. Yakimov
for interesting discussions and useful comments. Both authors were
partly supported by the NSF grant DMS-0070931 and by the CRDF
grant RM1-2244 and by the Swiss National Science Foundation.

\section{Tangles with flat $G$-connections
in their complements}\label{s1}

\subsection{Tangles and diagrams}\label{pigr}

Let $I\equiv[0,1]$ be the closed unit interval. A {\it geometric
tangle} $t\subset {\pl}\x I$ is the image of a  smooth embedding
$\underbrace{I\sqcup\ldots\sqcup I}_{k\ \mathrm{ times}} \sqcup
\underbrace{S^1\sqcup\ldots\sqcup S^1}_{l\ \mathrm{ times}} \to
\pl\x I$ with oriented components such that $\p t\subset {\pl}\x\p
I$, and $t$ intersects  the boundary $\pl\times \p I$
perpendicularly. We have $\p t=\p_+t \sqcup \p_-t$ where $\p_+
t=(\pl\times \{1\})\cap \p t$ and $\p_- t=(\pl\times \{0\})\cap \p
t$. The geometric tangle $t$ is called {\it standard} if $\p_+
t=\{(0,1,1),(0,2,1)\dots (0,n,1)\}$ and $\p_- t=
\{(0,1,0),(0,2,0)\dots (0,m,0)\}$ for some $n,m\in{\z}_{>0}$. We
will say that such tangle has type $(m,n)$.

Denote by $\Dt$ the image of $t$ under the projection
\begin{equation}\label{eq:proj}
p\colon {\pl}\x I\to {\mathbb R}\times I, \ (x,y,z)\mapsto(y,z)
\end{equation}
together with the additional information at each self-crossing point
about underpassing and overpassing segments.
The projection $\Dt$ is called {\it regular} if its only
singularities are double points at which
the images of the corresponding components of $t$
intersect transversally.
The ambient isotopy class of
a regular projection $\Dt$ is called a
\emph{diagram} of the tangle $t$.  When it is not confusing we
will use the same notation $D_t$ for
diagrams and regular projections.

The diagram of a tangle can be regarded as the isotopy class of a
 smooth embedding of a graph with four-valent
vertices of two types and one-valent vertices. These are the double points and
the boundary points, respectively. The edges of the graph are the segments
running between the vertices. A vertex of $D_t$ is called
\emph{positive} if it is a double point and the angle between
upper and lower components is positive. If the angle is negative,
the vertex is called \emph{negative}, see Fig.~\ref{pos-neg}. Each
such embedding is always a diagram of some tangle. Two diagrams
are called \emph{Reidemeister equivalent} if one can be obtained from
another by a finite sequence of Reidemeister moves
 (see Figs~\ref{fig-16}--\ref{fig-19}).

\begin{figure}[htbp]
\begin{center}
\subfigure[Negative intersection]{\scalebox{0.2}{\includegraphics{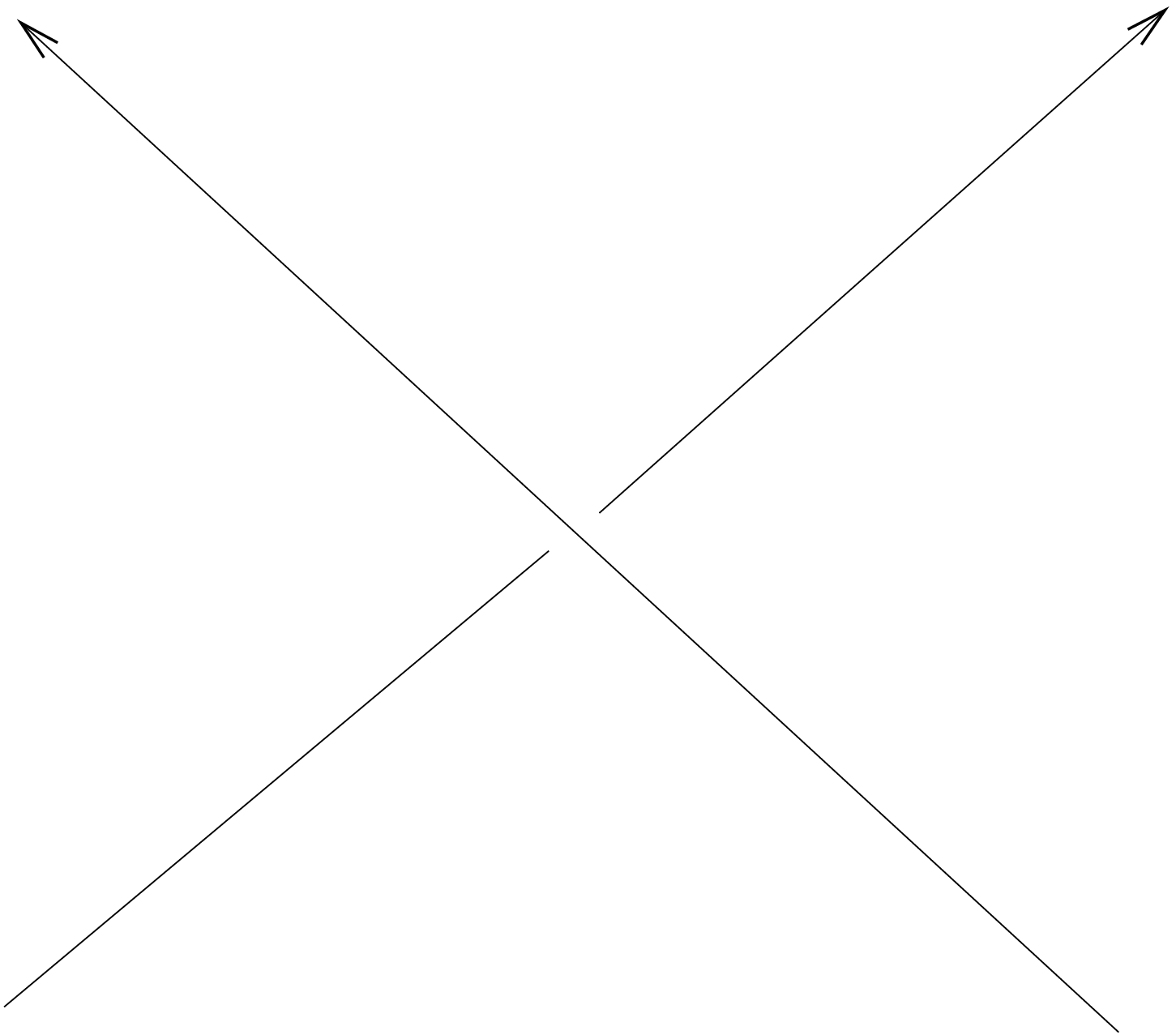}}}
\subfigure[Positive
intersection]{\scalebox{0.2}{\includegraphics{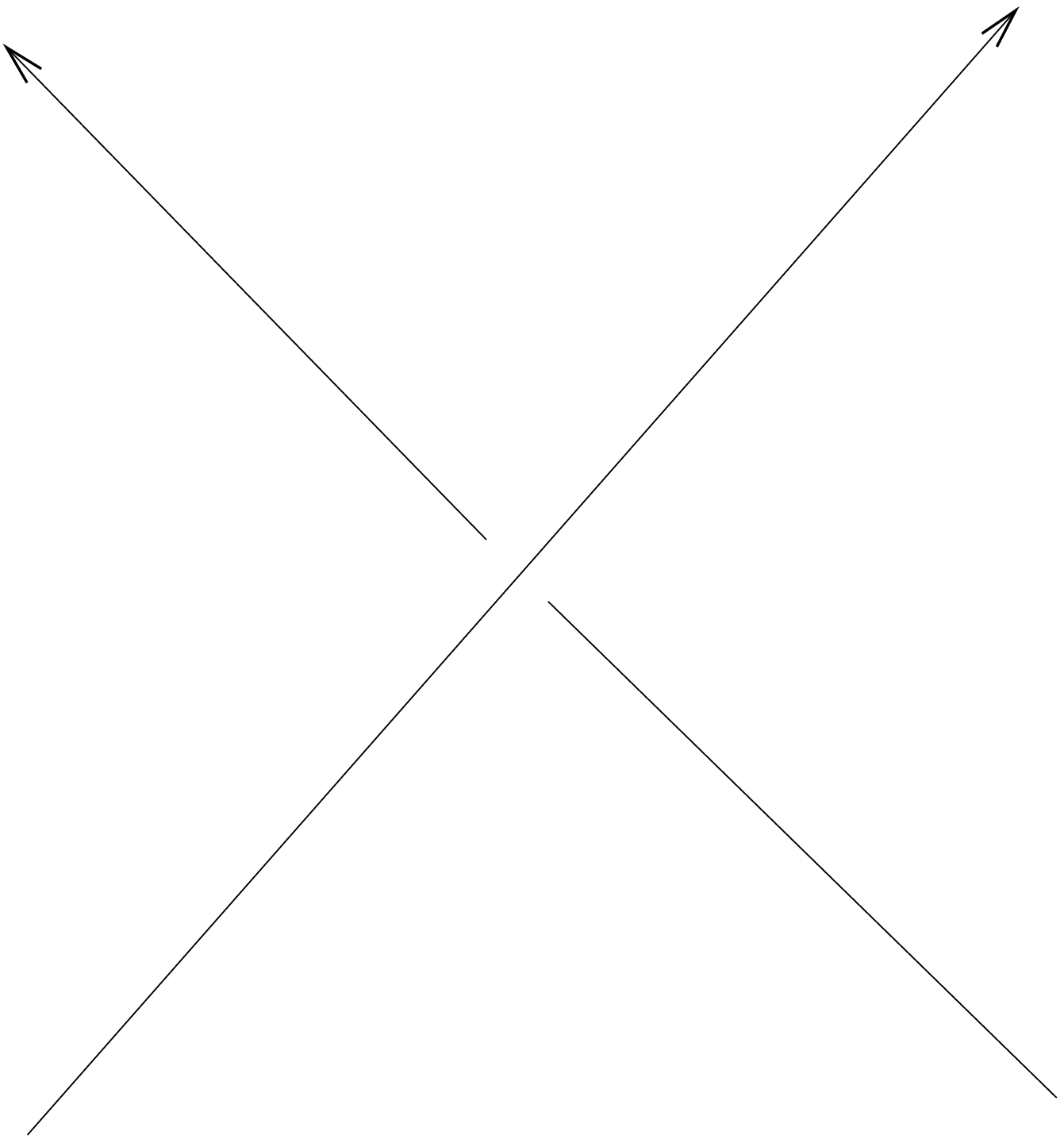}}}
\end{center}
\caption{}
\label{pos-neg}
\end{figure}

\begin{figure}[htbp]
\begin{center}
{\scalebox{0.4}{\includegraphics{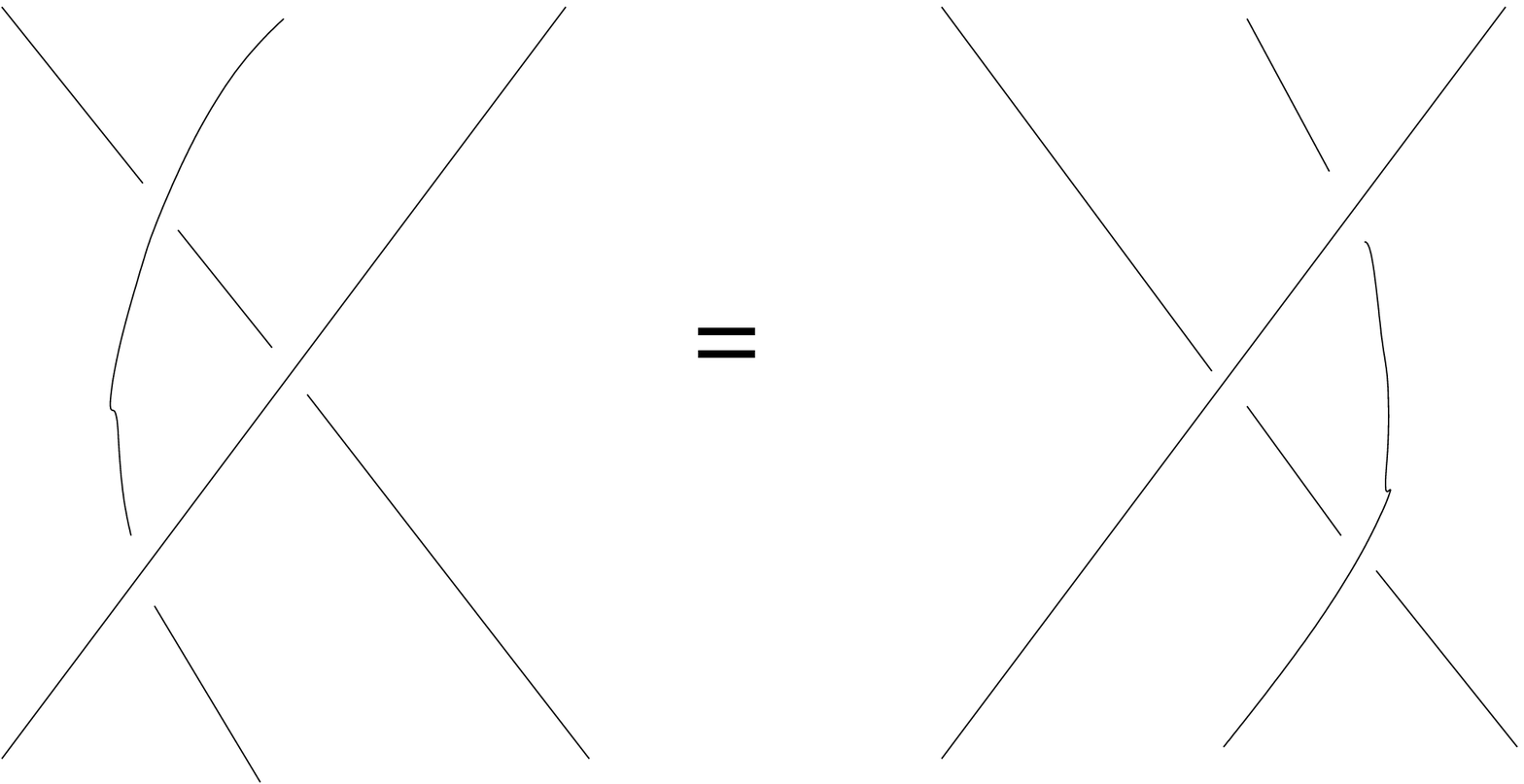}}}
\end{center}
\caption{}
\label{fig-16}
\end{figure}

\begin{figure}[htbp]
\begin{center}
{\scalebox{0.4}{\includegraphics{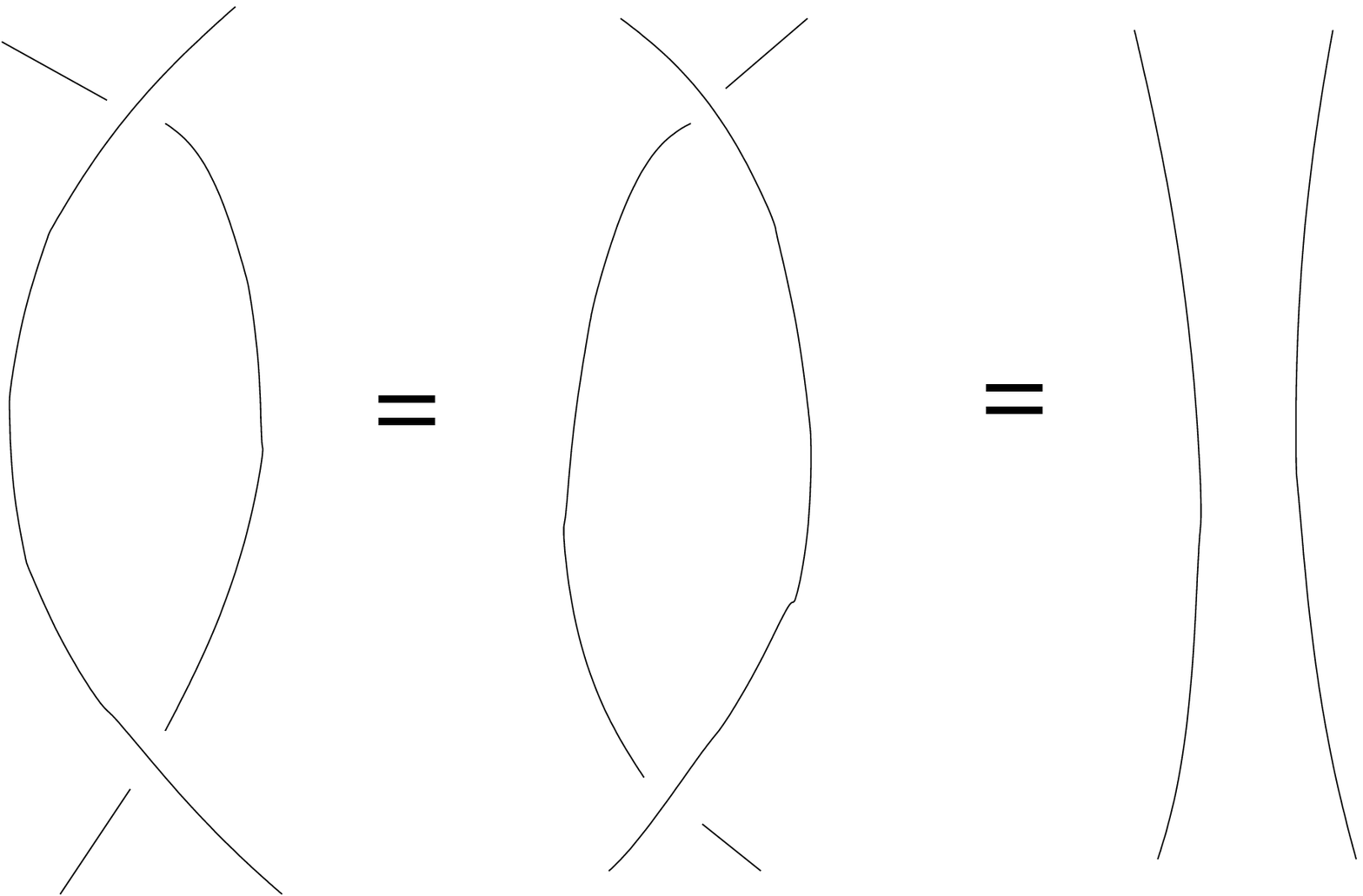}}}
\end{center}
\caption{}
\label{fig-17}
\end{figure}

\begin{figure}[htbp]
\begin{center}
{\scalebox{0.4}{\includegraphics{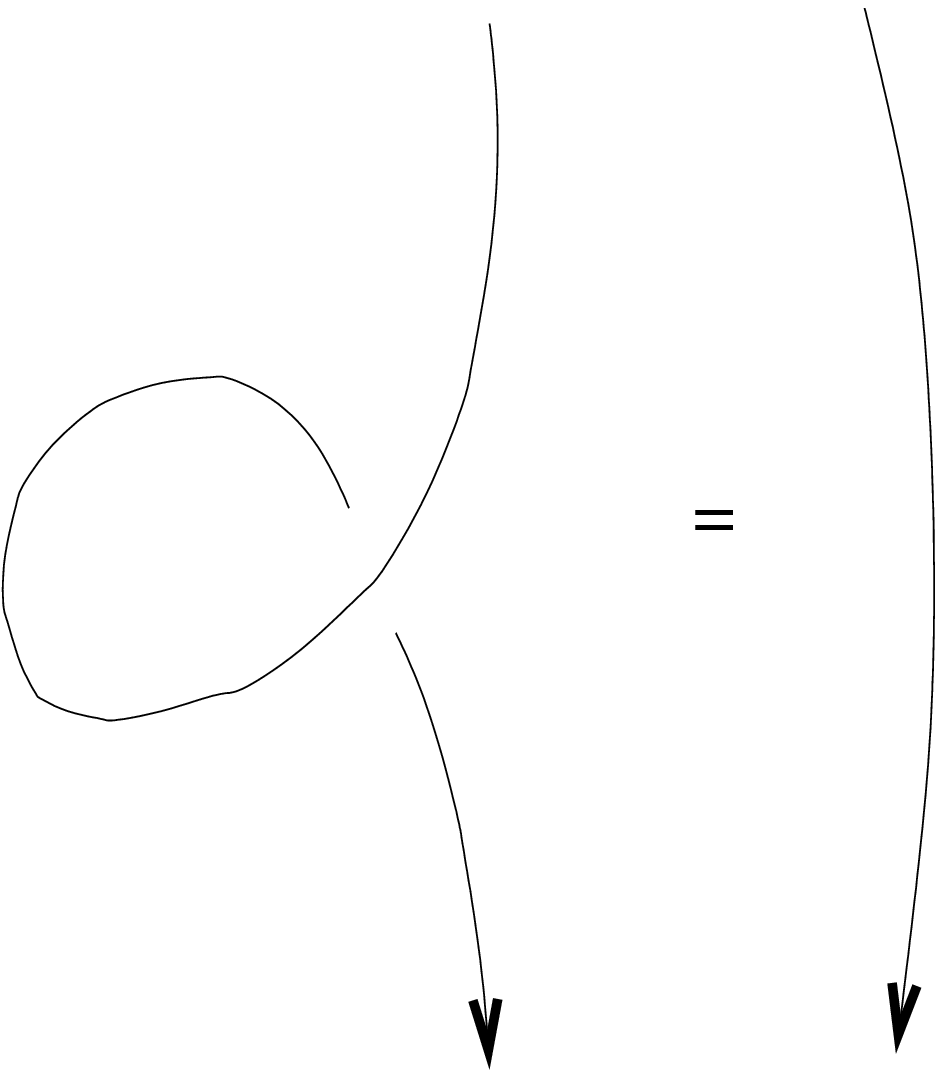}}}
\end{center}
\caption{\label{fig-19}}
\label{}
\end{figure}

A {\it tangle} is the isotopy class of a geometric tangle. Tangles are
in bijection with Reidemeister classes of diagrams.

A {\it geometric framing} of a geometric tangle $t$ is a
continuous nonsingular section of the normal bundle to $t$.
A framing of a tangle is the isotopy class of a
geometric framing. A framing is \emph{standard} if
it is parallel to the $x$-axis at $\partial_\pm t$ with
positive projection to this axis. It is clear that every framed
geometric tangle with such framing at the boundary is isotopically
equivalent to a framed geometric tangle with  the framing parallel
to $x$-axis. Framed tangle is the isotopy class of a framed
geometric tangle with the standard framing.

Two diagrams of tangles are called \emph{framed Reidemeister
 equivalent}
if they are connected by a sequence of moves in
Figs~\ref{fig-16},~\ref{fig-17},~\ref{fig-Rmove3}.

\begin{figure}[htbp]
\begin{center}
{\scalebox{0.4}{\includegraphics{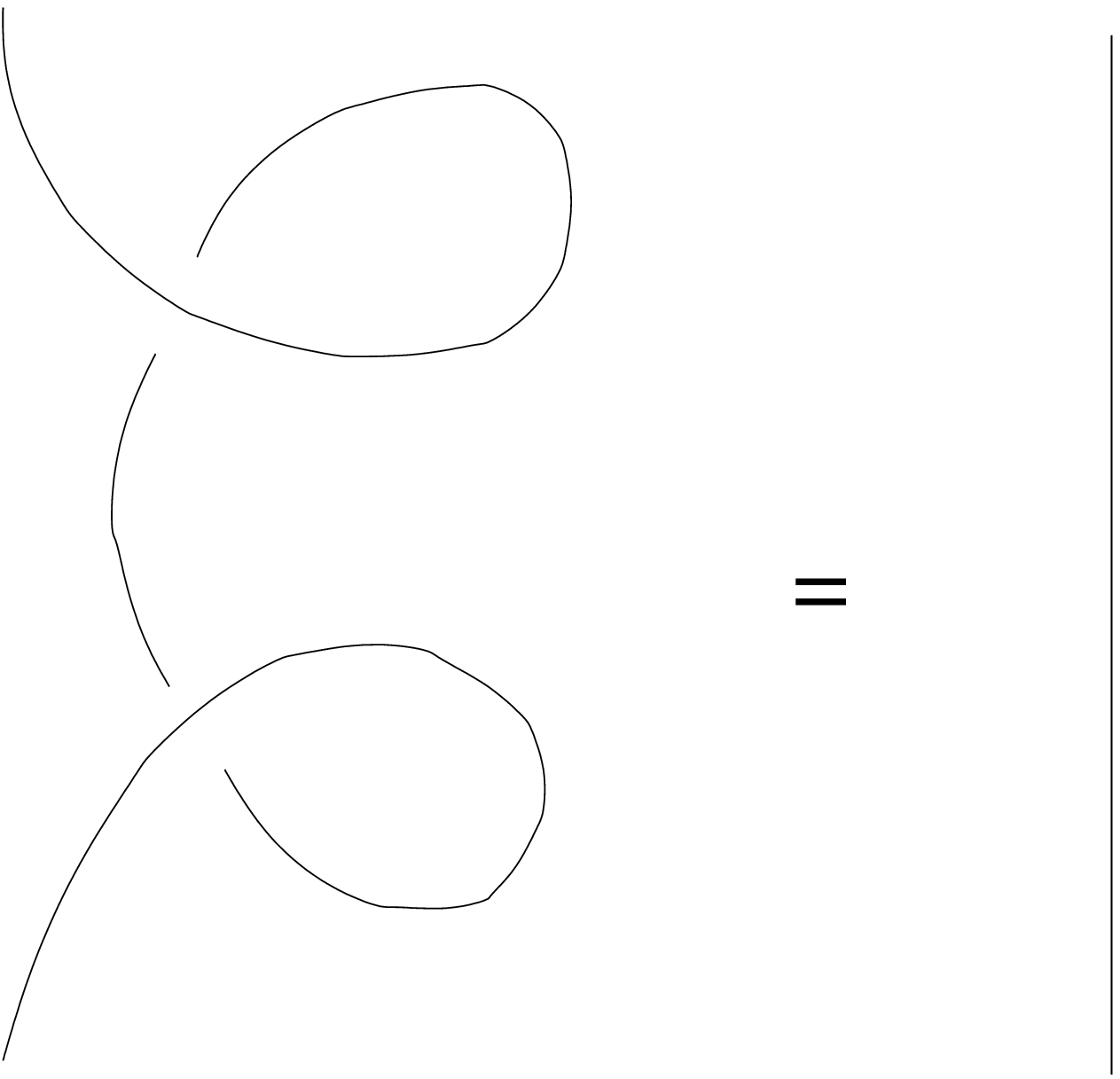}}}
\end{center}
\caption{}
\label{fig-Rmove3}
\end{figure}

Framed Reidemeister classes of diagrams are in bijection with
the framed tangles.

\subsection{Flat $G$-connections in the complement of a
tangle and $G$-tangles}

Let $E$ be a trivial principal $G$-bundle over $\pl\times I$ and
$A_t\in \Omega^1(\pl\times I\backslash t, {\mathfrak g})$ be the
1-form representing a flat connection in $E$ over the complement
to a standard geometric framed tangle $t$ of type $(m,n)$. The
corresponding  parallel transport operators along paths give an
equivalent description of the flat connection as a representation
of the fundamental groupoid  of $( \pl\times I) \backslash t$ in
$G$. \note{*}

Let $E_m$ be a trivial principal $G$-bundle over $\pl$,
 and $\alpha$ be the ${\mathfrak g}$-valued 1-form of  a flat connection in $E'$ over \note{*}
$\pl\backslash \{(0,1),\dots, (0,m)\}$.
Assume that $\alpha(x,y)$ decays sufficiently fast when $y$ goes to
$-\infty$ so that the parallel transport operator
of $\alpha$ along a path connecting
$(x_0,y_0)$ with $(x,y)$ in $\pl$  has a limit as  $y_0\to -\infty$. Let
$[\alpha]$ be the gauge class of $\alpha$ with respect
to gauge transformations trivial at $\infty$. Let $\gamma_i$ be a
path which starts at $(0,-\infty)$ encircles points $(0,1),\dots,
(0,i)$ and then returns to $(0,-\infty)$ (see Fig. \ref{lambda}).
We can identify the class $[\alpha]$ with an element of $G^m=G\times \dots
\times G$ so that the group element $g_i$ in the collection
$(g_1,\dots, g_m)\in G^m$ is the holonomy along the path $\gamma_i$.

\begin{figure}[htbp]
\begin{center}
{\scalebox{0.4}{\includegraphics{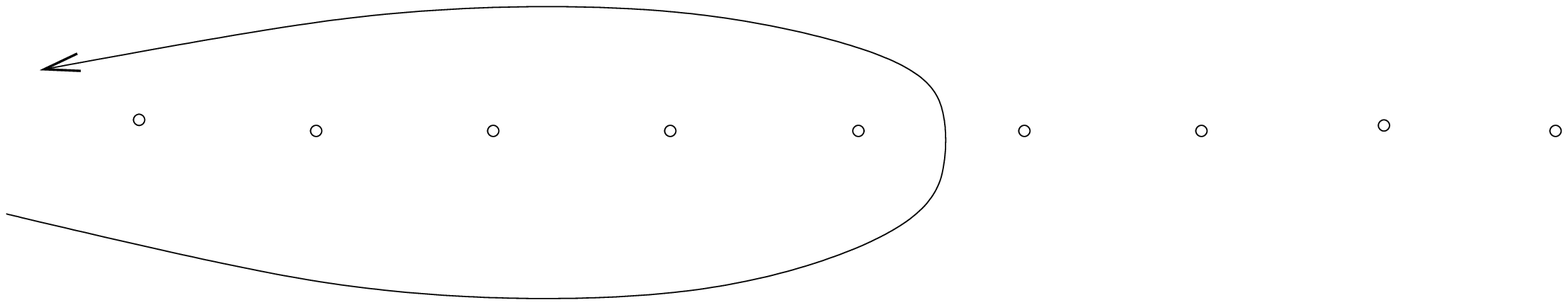}}}
\end{center}
\caption{}
\label{lambda}
\end{figure}

Let $\alpha$ and $\beta$ be flat connections in $E_m$ and $E_n$,
respectively. Denote by ${\mathcal A}_t(\alpha,\beta)$ the space
of flat $G$-connections over the complement of $t$ which, when
restricted to $\pl\times \{0\} \setminus\{(0,1,0),\dots,(0,m,0)\}$
and $\pl\times \{1\}\setminus\{(0,1,1),\dots,(0,n,1)\}$, coincide
with $\alpha$ and $\beta$, respectively, and which decay
sufficiently fast when $y\to -\infty$ (so that the parallel along
a path connecting $(x_0,y_0,z_0)$ and $(x,y,z)$ has a finite limit
when $y_0\to -\infty$).

If $G_t$ is the group of gauge transformations trivial at infinity,
the quotient space \\ ${\mathcal M}_t([\alpha], [\beta])= {\mathcal
A}_t(\alpha, \beta)/G_t$ is by definition the moduli space of those flat
$G$-connections in the complement of $t$ which, when restricted to
two boundary components $\pl\times \{0\}\setminus \partial t$
and $\pl\times \{1\}\setminus \partial t$, give
 representatives of the gauge classes $[\alpha]$ and $[\beta]$ respectively.

Following \cite{T-2}, denote $C_t=(\pl\times I)\backslash t$ and consider the
fundamental group of $C_t$ with the base point in $\{0\}\times
[-\infty, L]\times I$ with $t\subset \R\times
[L+1,\infty]\times I$. The set of such base points is contractible
and we will write $\pi_1(C_t)$ for the fundamental group,
suppressing the base point. We will call such base points left
base points. It is clear that an element of ${\mathcal
M}_t([\alpha], [\beta])$ defines a representation of $\pi_1(C_t)$.
Gauge classes of boundary values $[\alpha]$ and $[\beta]$ become
elements of $G^m$ and $G^n$ respectively, defined as holonomies
along $\gamma_i$ (see above). Thus, the space ${\mathcal
M}_t([\alpha], [\beta])$ can be identified
with a subspace of $G$-tangles \cite{T-2}, where a $G$-tangle is
defined as a pair consisting of
a framed oriented tangle $t$ and a group homomorphism
$\rho\colon\pi_1(C_t)\to G$.

Let $t$ and $t'$ be two standard isotopically equivalent geometric tangles.
An isotopy bringing $t$ to $t'$
lifts to an isomorphism between
$G$-connections over the complements of $t$ and $t'$ and therefore to
an isomorphism between corresponding moduli spaces of flat $G$-connections.
The isomorphism class of the moduli spaces generated by these isomorphisms
can be identified with the equivalence class
of pairs $(t, A_t)$ with respect to isotopies of tangles, their
pull-backs acting on connections and gauge transformations from $G_t$.
We will denote the obtained set of classes
by ${\mathcal M}_{[t]}([\alpha], [\beta])$.

From now on we will work with $G$-tangles. The geometric picture with
flat connections will be used only once below when
we shall assign a geometrical meaning to $G$-colorings of diagrams.

 Let
$v=(x,y,z)\in t\subset \pl\times I$ and $\gamma_v\subset C_t$ be a
(homotopy) path which starts at a left base point then goes to the point
$(x,y+\delta, z)$ "over" the tangle, then returns to the base
point "under" the tangle. Here $\delta$ is sufficiently small and
we assume that $t$ is transversal to $(x,y)$-plane  at $(x,y,z)$.
We will call such path (the homotopy class thereof) {\it standard} for
$v=(x,y,z)$.

\subsection{The category of $G$-tangles}
{\it Objects} of the category ${\mathcal T}(G)$ of $G$-tangles are
finite sequences $\{(\e_1, g_1),\dots, (\e_n, g_n)\}$ with
$\e_i=\pm 1$, \ $g_i\in G$.
{\it Morphisms } from $\{(\e_1, g_1),\dots, (\e_m, g_m)\}$ to \\
$\{(\sigma_1, h_1), \dots, (\sigma_n, h_n)\}$ are $G$-tangles.\\ Here
$m$ is the number of connected components of
$(\pl\times \{0\})\cap t$ and $n$ is the number of connected components of
$(\pl\times \{1\})\cap t$. The signs $\e_i$ and
$\sigma_i$ show the orientation of the boundary components. If $''+''$,
the component is oriented upward and if $''-''$, it is oriented downward.
The representation $\rho: \pi_1(C_t)\to G$ should agree with
$\{g_i\}$ and $\{h_i\}$  in the following way. If $\gamma^+_i$ is
a path from a left base point encircling boundary points $(0,1,0),
\dots (0,i,0)$ in the vicinity of $\pl\times \{0\}$ then $\rho(
\gamma^+_i)=g_i$. Similarly for a path $\gamma^-_i$ encircling
points $(0,1,1),\dots, (0,i, 1)$ in the vicinity of $\pl\times
\{1\}$ we have $\rho(\gamma^-_i)=h_i$.

The composition of morphisms is defined by gluing tangles.
The identity morphism of \\ $\{(\e_1, g_1),\dots, (\e_n, g_n)\}$ to
itself is the trivial braid with the representation of the
fundamental group defined by $g_1, \dots, g_n$.

\begin{rem}
The category ${\mathcal T}(G)$ can also be defined in more
geometrical terms of gauge classes of flat connections. {\it
Objects} are sequences $(\e_1,\dots,\e_n;[\alpha])$ where
$[\alpha]$ is a gauge class of a flat connection over $\pl
\backslash \{(0,1)\dots (0,n)\}$ in the trivial principal
$G$-bundle over $\pl$ and $\e_i=\pm 1$. {\it Morphisms} between
$(\e_1,\dots,\e_m;[\alpha])$ and $(\sigma_1,\dots, \sigma_n,
[\beta])$  are elements of $ {\mathcal M}_{[t]} ([\alpha],
[\beta])$  i.e. equivalence classes of pairs $(t, A_t)$ described
in the previous subsection.
\end{rem}

\section{The category of $G$-colored diagrams}\label{s3}
\subsection{Factorizable groups and Lie groups}

We say that group $G$ is {\it factorizable}
 into two subgroups $G_\pm\subset G$ if any element $g\in G$
can be represented in a unique way as
\begin{equation}\label{la-factor}
g=g_+g_-^{-1}
\end{equation}
where $g_\pm\in G_\pm$.

If $G$ is a complex algebraic Lie group (later we will focus on
this case) we will say that it is factorizable if there exists a
Zariski open neighborhood $G'\subset G$ of $1$ such that every
element of $G'$ has a unique factorization (\ref{la-factor}).
Notice that in this case any $l\in \mathfrak{g}=Lie(G)$ has a
unique decomposition
\[
l=l_+ -l_-
\]
where $l_\pm\in \mathfrak{g}_\pm=Lie(G_\pm)$.

\begin{rem} We can choose
$G_+=\{e\}$ and $G_-=G$. We call it trivial factorizability.
\end{rem}

Let $G$ be a factorizable group. Define a binary
operation
\[
g\star h=g_+h_+(g_-h_-)^{-1}
\]
which obviously defines a group structure on $G$ with the
same identity element as for the original group structure. The
inverse of $g$ in this group is $i(g)=g_+^{-1}g_-$. This operation
corresponds to the multiplication of the group $G_+\times G_-$
under the mapping
$G_+\times G_- \to G$, \ $(g_+,g_-)\mapsto g_+(g_-)^{-1}$. In what
follows, we shall denote this group $G^*$.

\subsection{$G$-colorings of diagrams}
Let $t$ be a standard geometric tangle and $D_t$, its diagram with
the set of edges $E(D_t)$.
Assume that $G$ is a factorizable group.

\begin{defin}\label{fact-rel}
The map $E(D_t)\to G$ which associates to edge $e$
element $x_e\in G$ is called a $G$-coloring of
diagram $D_t$ if at each double point it  satisfies the relations
$$
  x_{b_v} = (x_{a_v})^{-1}_\pm x_{c_v} (x_{a_v})_\pm\, , \quad
    x_{a_v} = (x_{c_v})_\mp x_{d_v} (x_{c_v})_\mp^{-1}\, ,
$$
depending on whether the intersection is  positive or negative.
Here the enumeration of edges is the same as on
Fig.~\ref{fig-enum}.
\end{defin}

\begin{figure}[htbp]
\begin{center}
{\scalebox{0.4}{\includegraphics{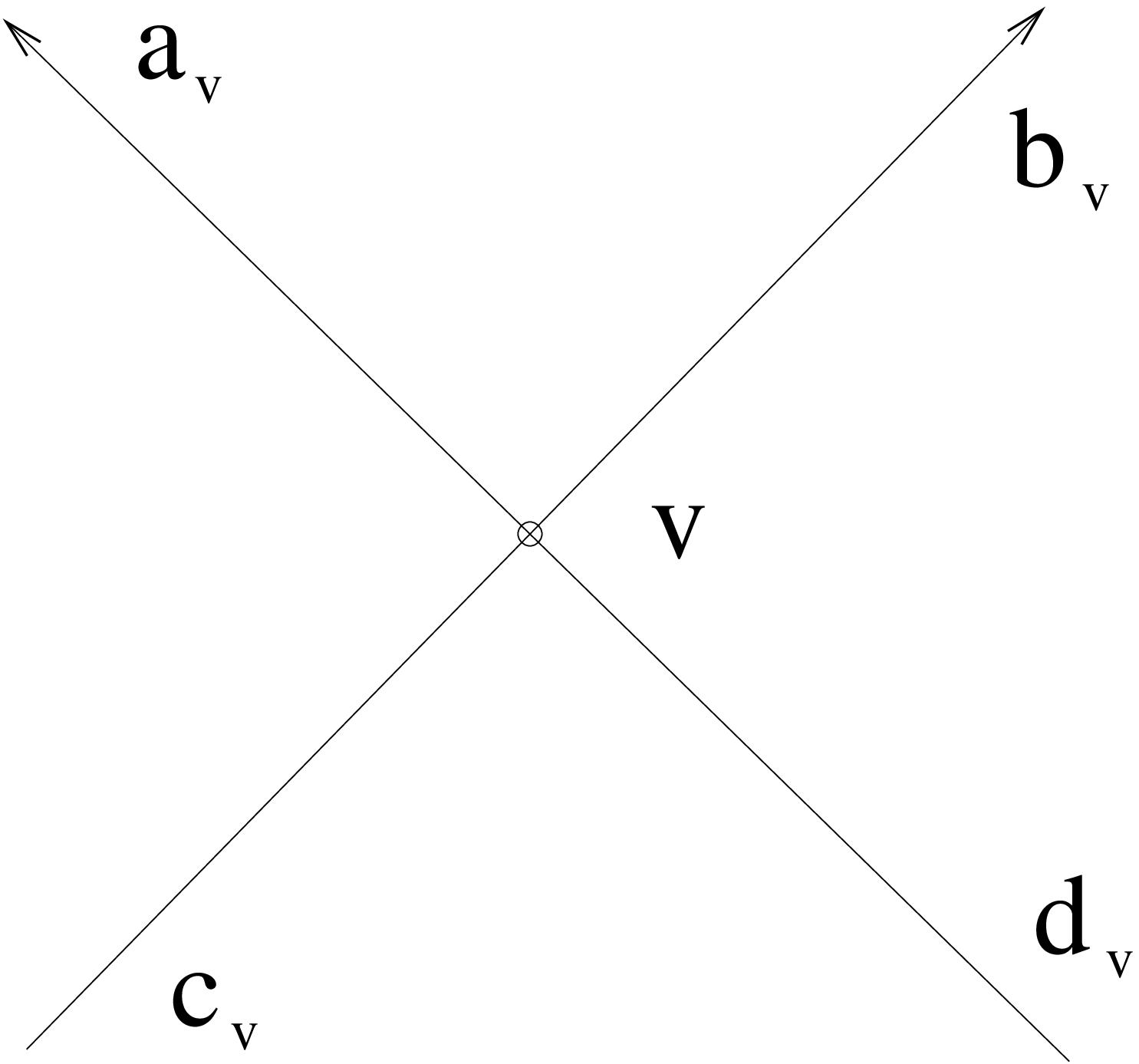}}}
\end{center}
\caption{}
\label{fig-enum}
\end{figure}

Let $x_L, x_R\colon G\times G\to G$ be mappings acting as
\begin{equation}\label{xLR}
x_L(x,y)=x_-yx_-^{-1}, \quad  x_R(x,y)=x_L(x,y)_+^{-1}x x_L(x,y)_+
\end{equation}

In terms of these maps the definition above means that at positive
double points we have $(x_a,x_b)=(x_L(x_c,x_d),x_R(x_c,x_d))$
and at negative double points $(x_c,x_d)=(x_L(x_a,x_b),x_R(x_a,x_b))$.

The following proposition is due to  Weinstein and Xu \cite{WX}
in the context of factorizable Poisson--Lie groups.

\begin{prop} \label{mathcalR}The map ${\mathcal R}\colon
G\times G\to G\times G$
acting as $(x,y)\mapsto
(x_L(y,x), x_R(y,x))$ satisfies the set-theoretical Yang--Baxter
equation:
\[
{\mathcal R}_{12}{\mathcal R}_{13}{\mathcal R}_{23}
={\mathcal R}_{23}{\mathcal R}_{13}
{\mathcal  R}_{12}
\]
\end{prop}
Here all mapping act from $G^{\times 3}$ to itself. The mapping
${\mathcal R}_{12}$ acts as $\mathcal R$ in the first two factors
and trivially in the last one. The other mappings act in a similar
way. The subindices indicate in which factors the mapping acts
non-trivially.

It follows from this proposition that if the $G$-coloring of lower edges of
both diagrams on
Fig.~\ref{fig-16} are $x, y, z$, then the colorings of upper edges
of the diagrams (which are determined by the diagrams and the coloring
of lower edges) are the same for both diagrams. In other words, the
$G$-coloring
is compatible with the third Reidemeister move. It is easy to see that
the $G$-coloring is also compatible with other framed Reidemeister moves.

Let $(D,c)$ and $(D',c')$ be two $G$-colored diagrams
which are Reidemeister equivalent. Then, since in each Reidemeister
move the coloring of a  new
diagram is uniquely defined by the coloring of the initial
diagram, $c'$ is uniquely determined by $c$.
Thus, we have Reidemeister classes of $G$-colored diagrams.

\subsection{The category of $G$-colored diagrams}\label{ss4}

Let $G$ be a factorizable group.
Define category ${\mathcal D}(G)$ of $G$-colored diagrams  as
follows.

{\it Objects} of the category are sequences $\{(\e_1,x_1),\dots,(\e_n,x_n)\}$
where $\e_i=\pm$ and $x_i\in G$ and the empty set.

{\it Morphisms} between  $\{(\e_1,x_1),\dots,(\e_n,x_n)\}$ and
$\{(\sigma_1,y_1),\dots,(\sigma_m,y_m)\}$ are Reidemeister classes
of $G$-colored diagrams with the orientation of the boundary edges
(adjacent to 1-valent vertices) defined by $\e_i$ and $\sigma_j$
as it is shown on Fig. ~\ref{fig-10} and with the $G$-colorings of
the boundary edges given by $\e_i(x_i)$ and $\s_j(y_j)$ where
$\e(x)$ is defined as
\begin{equation}\label{invex}
\e(x)=\left\{ \begin{array}{cc} x=x_+x_-^{-1} & \mbox{for $\e=+1$}
\\ i(x)= x_+^{-1}x_- & \mbox{for $\e= -1$}
\end{array}
\right.
\end{equation}
The operation $x\to i(x)$ is taking the inverse in $G^*$. The
identity morphism is shown on Fig.~\ref{fig-11}.

{\it Composition} of Reidemeister classes of $G$-colored
diagrams $(D_1,c_1)$ and $(D_2,c_2)$ is the
Reidemeister class of the $G$-colored diagram $(D_1\circ D_2, c)$
where $D_1\circ D_2$ is the diagram obtained by gluing $D_1$
and $D_2$ and the coloring $c$ is induced by colorings $c_1$
and $c_2$, 

\begin{figure}[htbp]
\begin{center}
{\scalebox{0.4}{\includegraphics{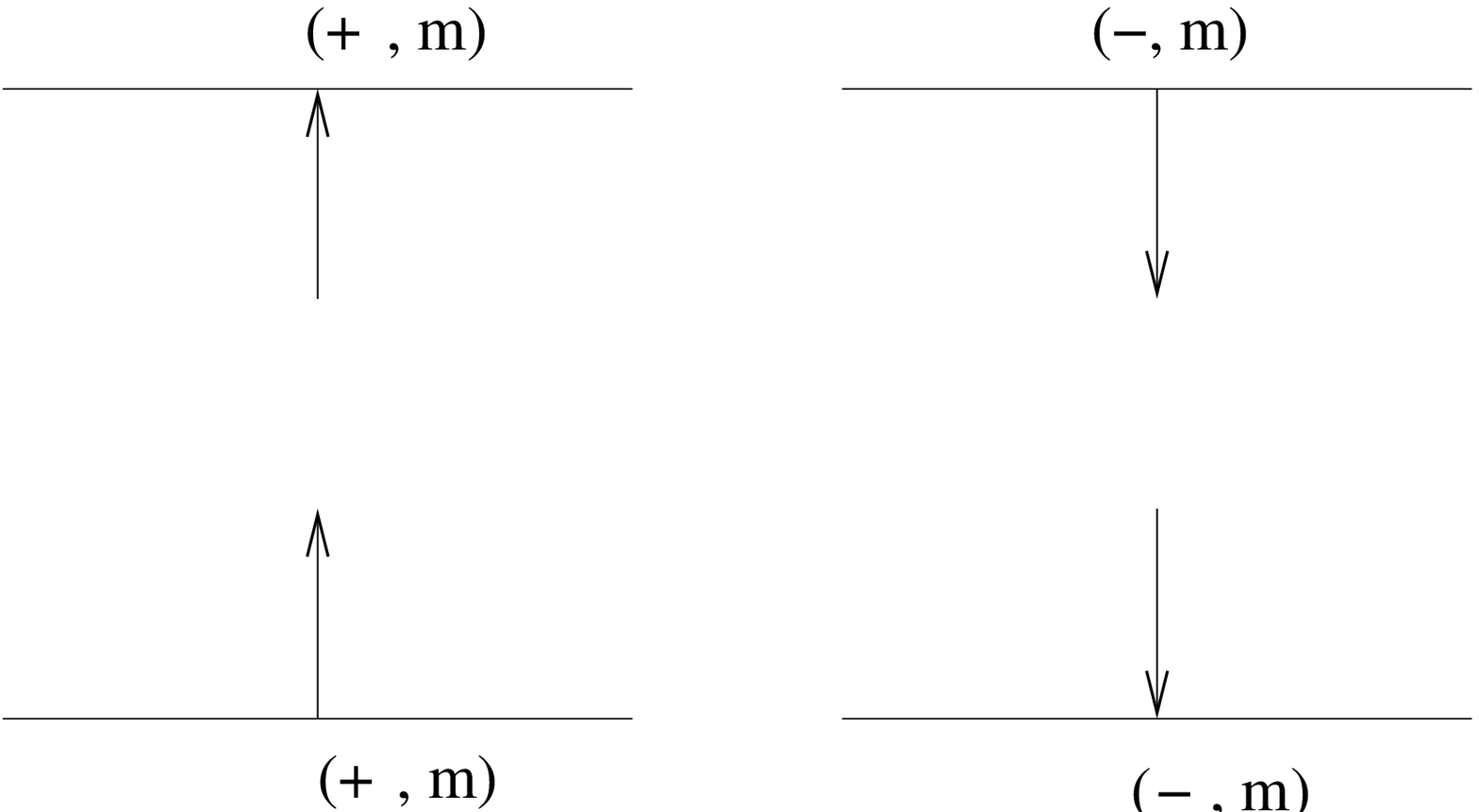}}}
\end{center}
\caption{}
\label{fig-10}
\end{figure}

\begin{figure}[htbp]
\begin{center}
{\scalebox{0.4}{\includegraphics{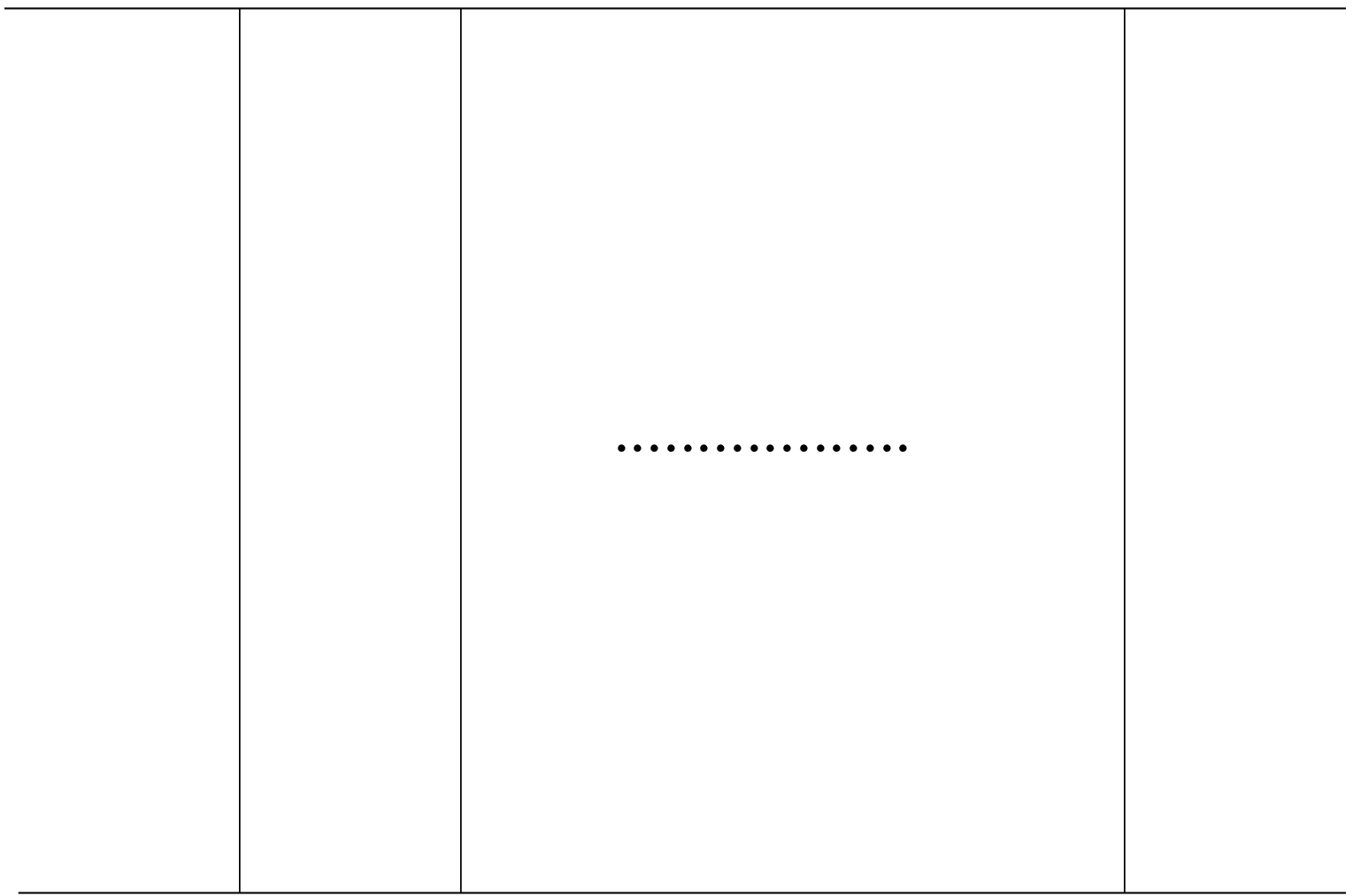}}}
\end{center}
\caption{}
\label{fig-11}
\end{figure}

\subsection{The equivalence of categories}
Here we will prove the equivalence of categories
${\mathcal T}(G)\simeq {\mathcal D}(G)$. Consider the map $F:
{\mathcal T}(G)\to {\mathcal D}(G)$ acting on objects as
\[
F((\e_1,g_1),\dots, (\e_n,g_n))=\{(\e_1, x_1)\dots (\e_n, x_n)\}
\]
Here $x_i\in G$ are related to $g_i\in G$ via
\begin{equation}\label{g-x}
g_i=(x_1)_+^{\e_1}\dots(x_i)_+^{\e_i}(x_i)_-^{-\e_i}(x_1)_-^{-\e_1}
\end{equation}

For a $G$-tangle $(t,\rho)$ define
\[
F((t,\rho))=[(D_t,c)]
\]
where $D_t$ is a diagram of the tangle $t$, $c$ is the coloring of
$D_t$ which we define below, and $[(D_t,c)]$ the colored framed
Reidemeister class of $(D,c)$.

Consider a standard
path $\gamma_v$ associated with point $v\in t$. Let
$e_1,\ldots,e_i$ be the edges of $D_t$
intersected by the projection of $\gamma_v$, and $\e_1,\ldots, \e_i$
the signs of the projections of their orientations to the vertical
axis. Then  the holonomy $g_v=g_i$ along $\gamma_v$ associated to
$\rho$ is given by formula~\eqref{g-x}, where $x_1,\ldots, x_i$ now
are the colors of edges $e_1,\ldots,e_i$. It is easy to see that this
correspondence between $\rho$ and the edge colors is one to one, and
thus the mapping $F$ is invertible.

It is easy to see that the map $F$ a functor.
To prove this it remains to show that $F(fg)=F(f)F(g)$
for morphisms $f$ and $g$, which is obvious.
It is also clear that this functor is an equivalence of
categories.

\subsection{ A geometric version of the functor $F$} Now consider a
geometric description of the category ${\mathcal T}(G)$ and describe
the functor $F$ in this terms.

To define the action of $F$ on objects we consider a special representative
$\tilde{\alpha}$ of the gauge class $[\alpha]$. Namely,
$\tilde{\alpha}$ is continuous and vanishes outside the  strips
$\{(x,y)|\
x\in {\mathbb R},\ i-\epsilon<y<i+\epsilon\}$ for some
$0<\epsilon<1/2$. Let $\gamma_i^{\pm}$ be paths connecting points
$(0,i-1/2)$ and $(0, i+1/2)$ which go around $(0,i)$ in the
clock-wise direction for $\gamma^-_i$ and in the counter clock-wise
direction for $\gamma^+_i$. These paths are shown as $(-)$ and $(+)$
paths respectively on  Fig. \ref{Paths}.

\begin{figure}[htbp]
\begin{center}
{\scalebox{0.4}{\includegraphics{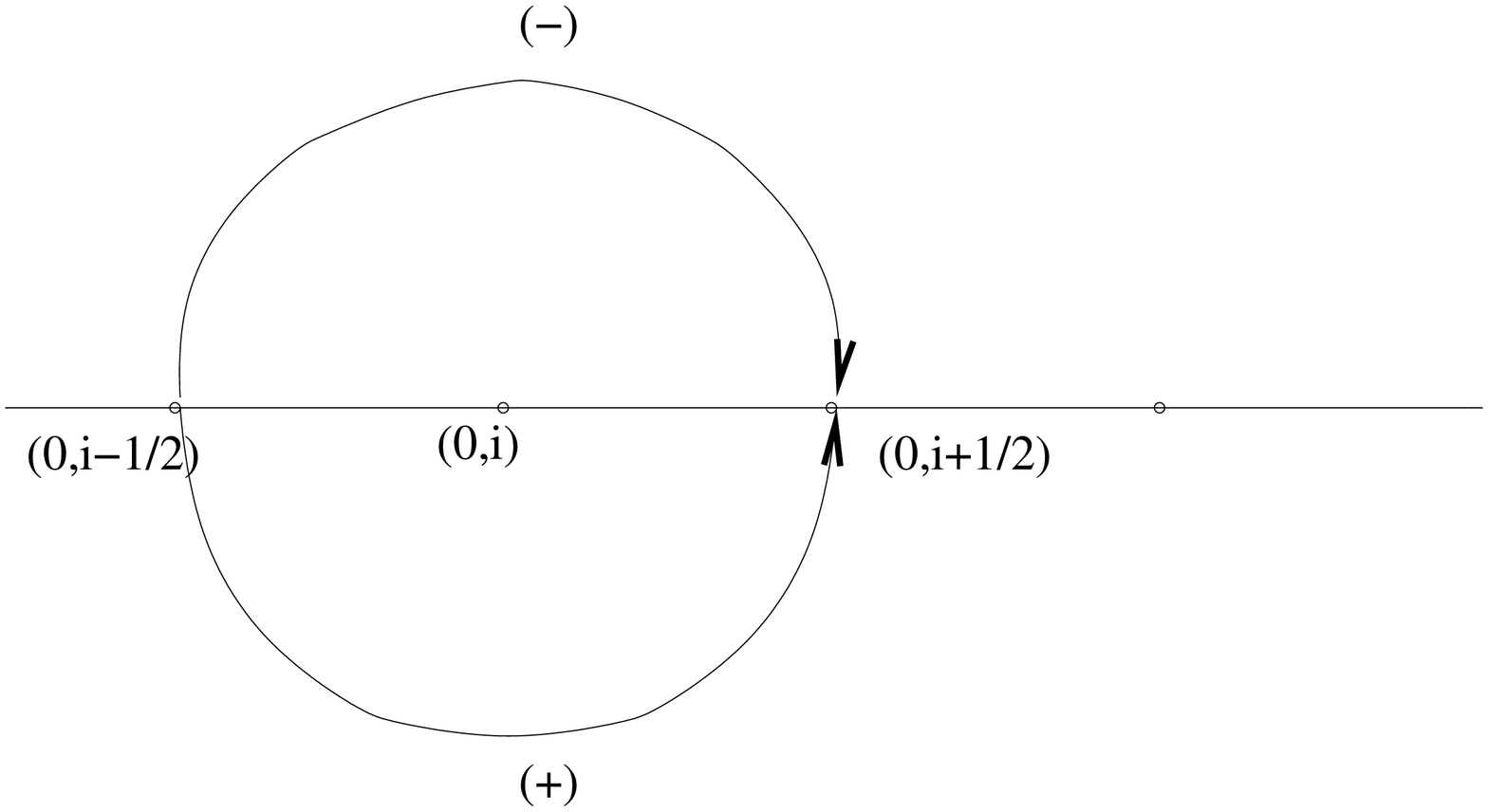}}}
\end{center}
\caption{}
\label{Paths}
\end{figure}

Define the action of $F$ on objects as:
\[
F:(\e_1,\dots,\e_n; [\alpha])\mapsto \{(\e_1,x_1),\dots, (\e_n,x_n)\}
\]
where $x_i=hol_{\gamma_i}(\tilde{\alpha})$ and
$\gamma_i=\gamma_i^+\circ(\gamma_i^-)^{-1}$.

To define the action of $F$ on morphisms, first,
let us look at the geometry of the projection $p$. The preimages
of edges of $D_t$ form a system of ``walls'' $p^{-1}(e)\subset \pl\times I,
\ e\in E(D_t)$, intersecting at
lines which are preimages of the vertices of $D_t$.
Let $A_t$ be a flat connection over the complement of $t$. If the
tangle is not a link all chambers bounded by walls including the
"outer" chambers are simply connected so the flat connection $A_t$
can be trivialized inside these chambers, except a thin
neighborhood of walls. If the tangle is a link, the outer chamber
is not simply connected, but we still can trivialize the flat
connection inside this chamber since the link can be placed into a
ball and the outside of a ball is a simply connected region in
$\pl\times I$. Thus, we can trivialize the flat connection in each
chamber outside of a thin neighborhood of walls. Assume that this
thin neighborhood is such that at the boundary it is inside of
strips $x\in \R, \ i-\epsilon <y<i+\epsilon$.
The tangle $t$ separates each wall into two semi-infinite parts.
Let $e$ be an edge of $D_t$, denote by $p^{-1}(e)_-$ the part of
the wall $p^{-1}(e)$ which is semi-infinite in the negative
$x$-direction and by $p^{-1}(e)_+$ the other part of this wall.
Since the connection is flat and since now it is trivial inside
the chambers (except of a thin neighborhood of walls), the
holonomy through the wall $p^{-1}(e)$ along any path that is based
on a pair of points separated by this wall and which are outside
of a thin neighborhood of the wall, depends only on the homotopy
class of the path. Let us call the path positive if its
orientation together with the orientation of the edge $e$ and with
the direction of the projection $p$ form positively oriented
triple of vectors in ${\mathbb R}^3$ with the standard
orientation. This produces two elements $x_\pm(e)$ of $G$ which we
can assign to the edge $e$ corresponding to holonomies along
positive paths through $p^{-1}(e)_\pm$.
The product $x(e)=x_+(e)x_-(e)^{-1}$ is the holonomy along a
closed path that crosses first the wall $p^{-1}(e)_+$ in the
positive direction and then $p^{-1}(e)_-$ in the negative
direction. If $G$ is factorizable there exists unique pair
$x(e)_\pm$ which factorizes $x(e)$ as above with $x(e)_\pm \in
G_\pm$, and one can choose the connection $A_t$ so that $x_\pm(e)=x(e)_\pm$.
This gives us an assignment $e\mapsto x(e)$ of group elements to
edges of the diagram. Notice that this assignment does not depend
on base points inside chambers and depends only on the gauge class
of the flat connection.

The holonomy along a path connecting two points based inside
chambers is the product of holonomies through the walls
intersected by this path. Given a tangle $t$ and its diagram
$D_t$ these holonomies can be computed by the {\it wall crossing
rule} (see Fig. \ref{wall-cr}). Horizontal edges on Fig.~\ref{wall-cr}
represent parts of the path. Vertical edges
represent parts of the tangle (walls). Under-crossings and
over-crossings show whether the path went through $p^{-1}(e)_-$ or
$p^{-1}(e)_+$  respectively where $e$ is the corresponding edge of
the diagram. The holonomy gained at the crossing is given in
terms of the $G$ coloring of $e$.

To show that the map assigning group elements to edges constructed
above gives $G$-colorings, one should consider pairs of paths from
Figs~\ref{mon-1}--\ref{mon-4}. Isotopy equivalence of these paths
implies the equality of the corresponding holonomies. Computation
of them according to the ``wall-crossing rules'' described above
gives the identities in the definition of the $G$-coloring.

\begin{figure}[htbp]
\begin{center}
{\scalebox{0.4}{\includegraphics{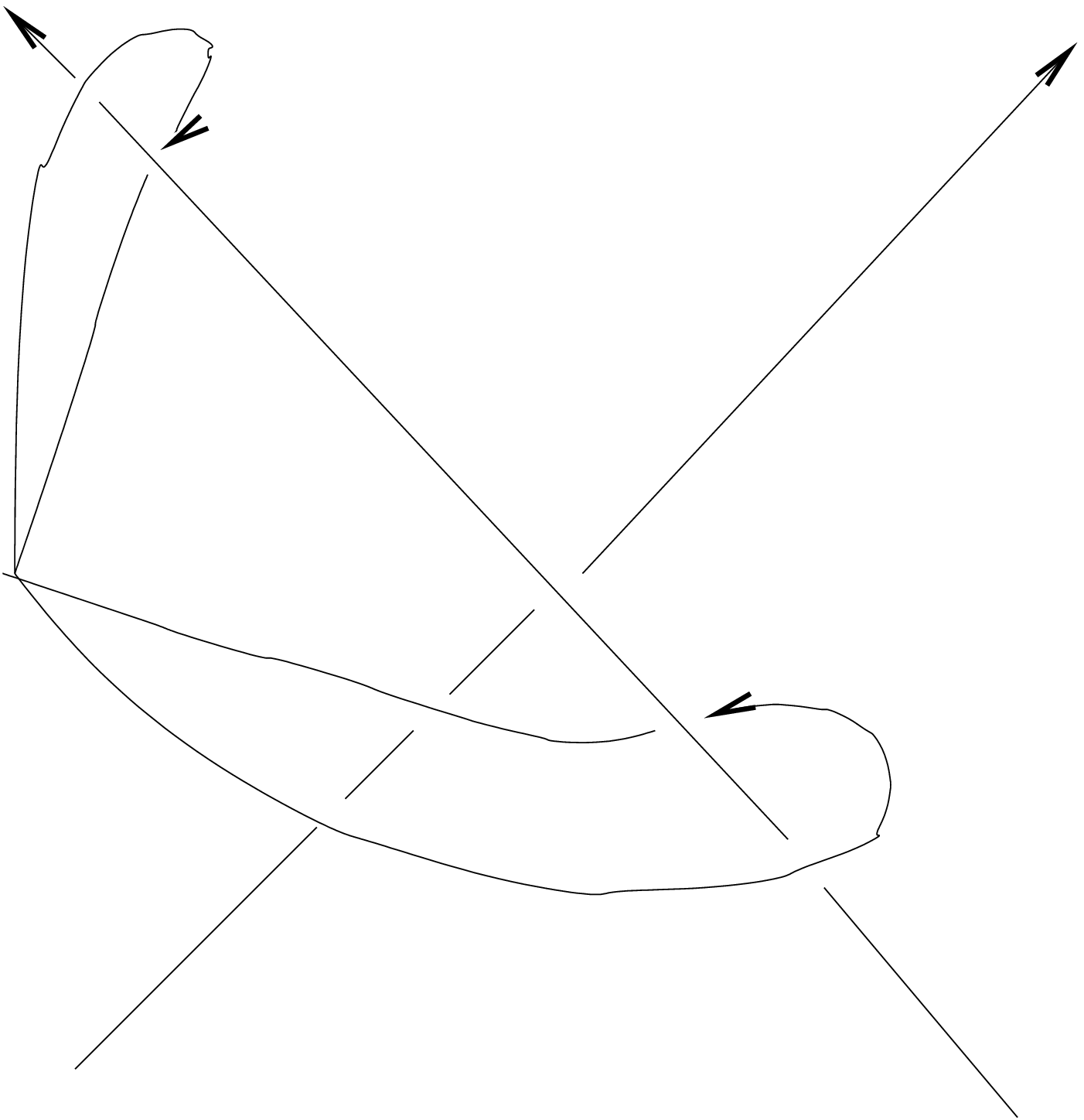}}}
\end{center}
\caption{}
\label{mon-1}
\end{figure}

\begin{figure}[htbp]
\begin{center}
{\scalebox{0.4}{\includegraphics{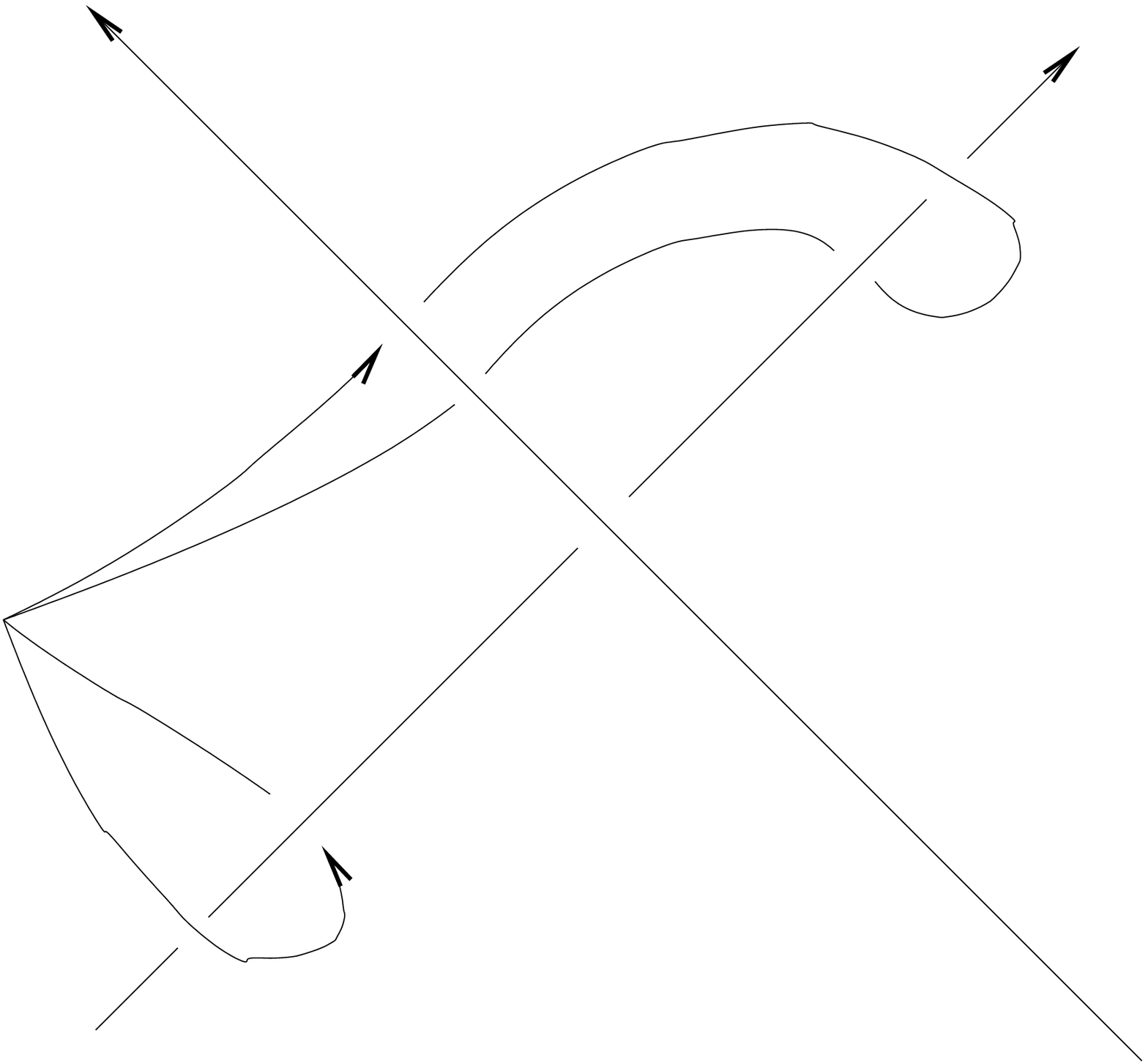}}}
\end{center}
\caption{}
\label{mon-2}
\end{figure}

\begin{figure}[htbp]
\begin{center}
{\scalebox{0.4}{\includegraphics{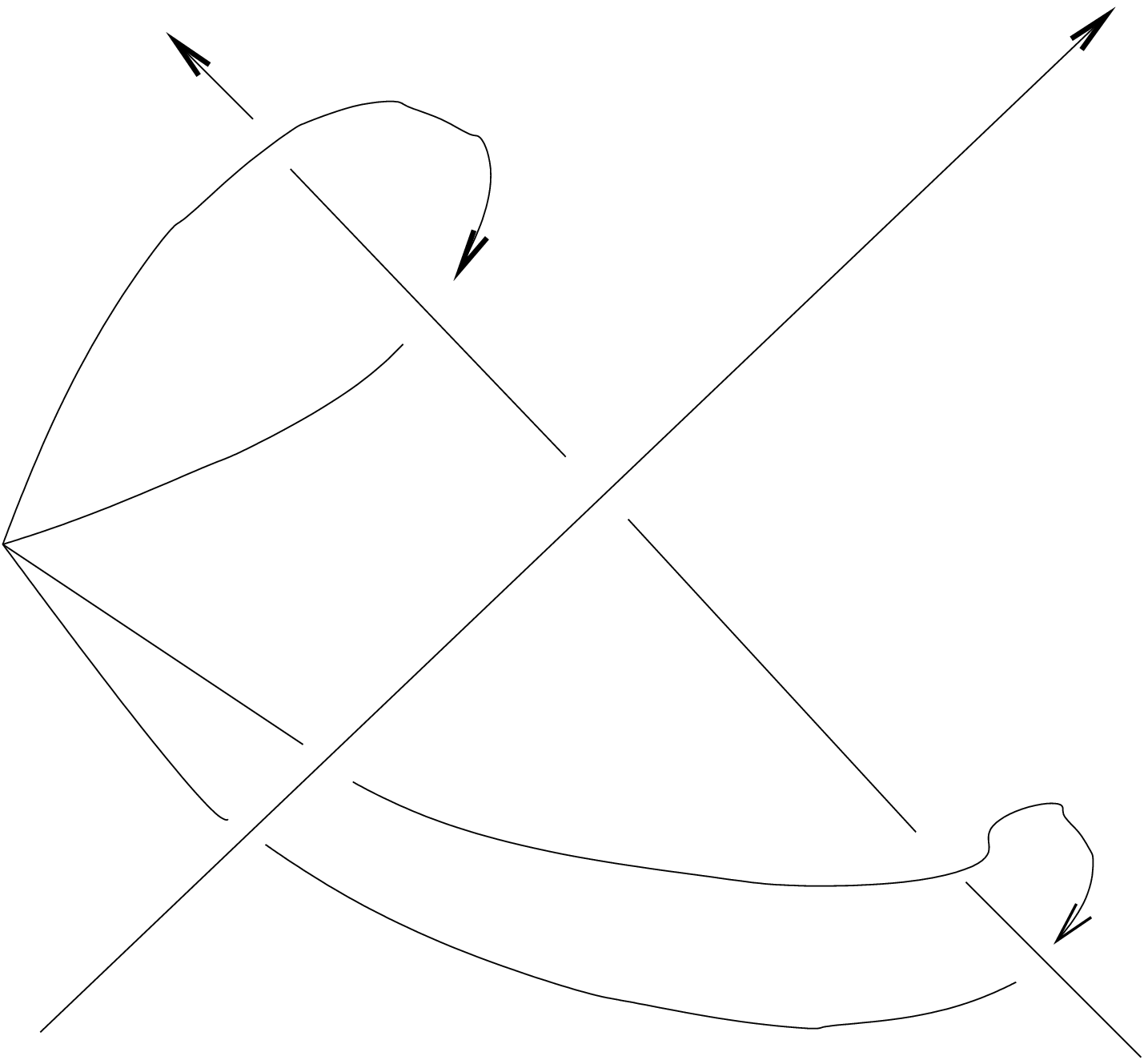}}}
\end{center}
\caption{}
\label{mon-3}
\end{figure}

\begin{figure}[htbp]
\begin{center}
{\scalebox{0.4}{\includegraphics{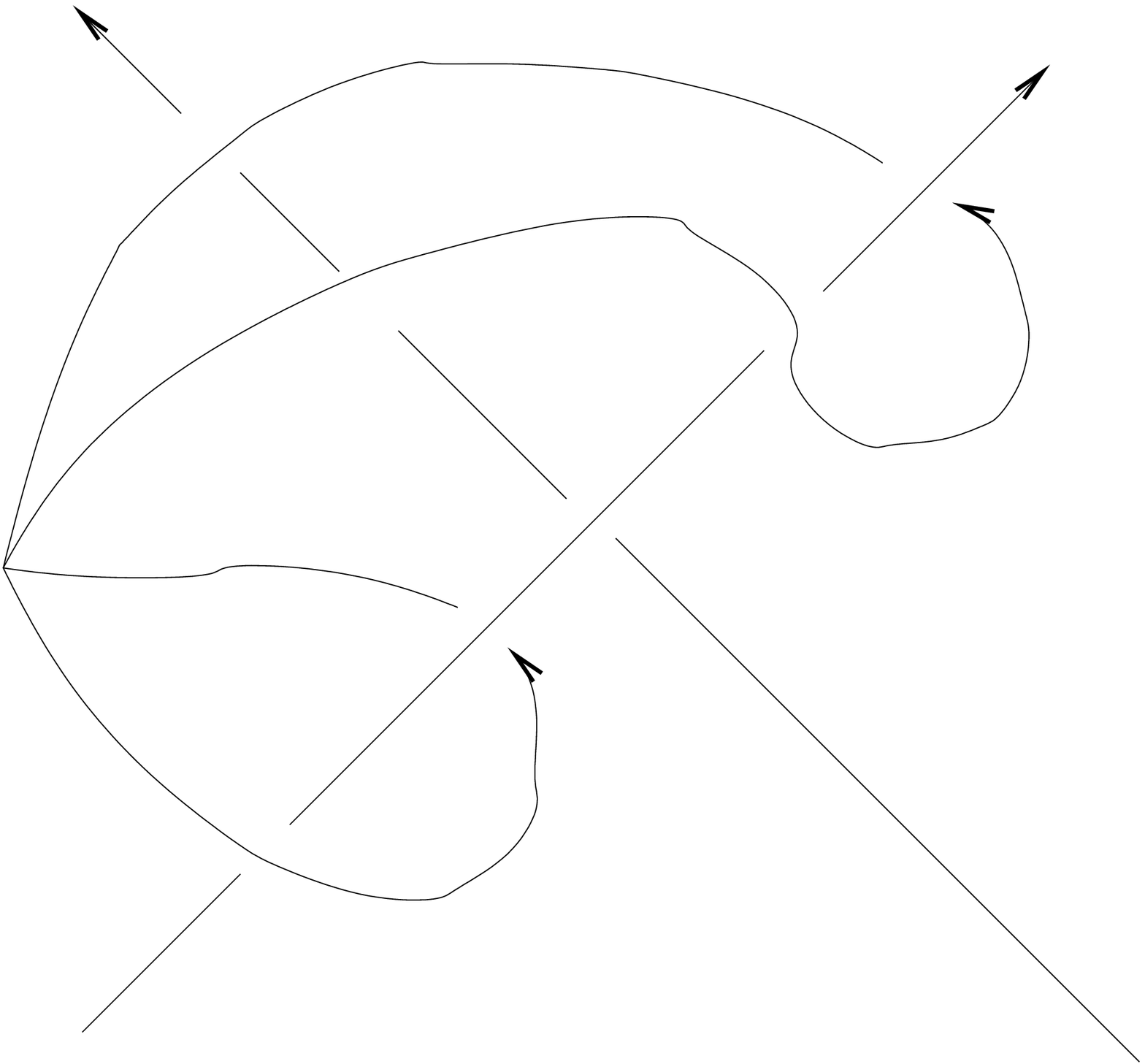}}}
\end{center}
\caption{}
\label{mon-4}
\end{figure}

\begin{figure}[htbp]
\begin{center}
{\scalebox{0.4}{\includegraphics{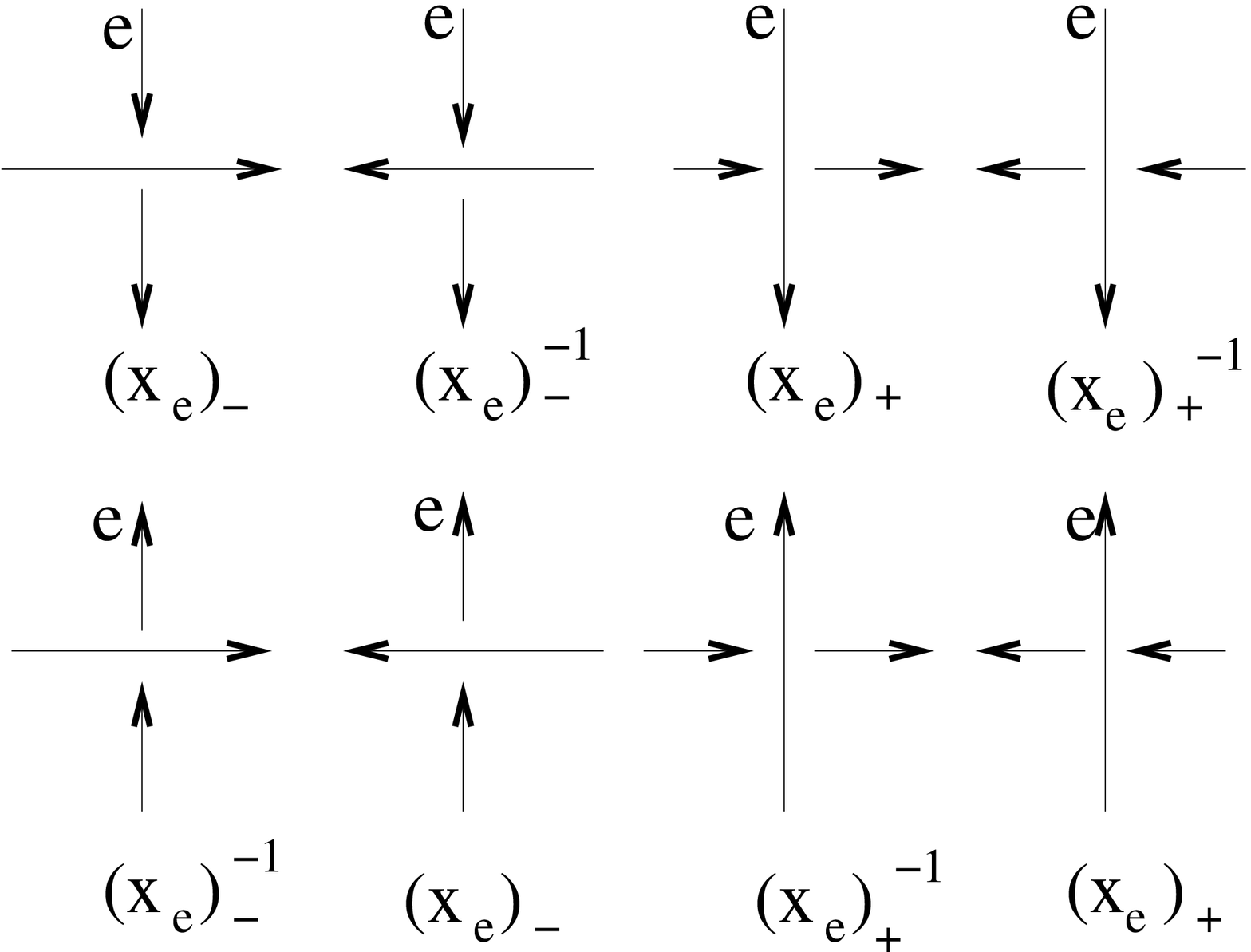}}}
\end{center}
\caption{}
\label{wall-cr}
\end{figure}

\section{Braided $G$-categories}\label{s31}

\subsection{Braided $G$-categories}

A  \emph{braided group} is a pair $(G,\rc: G \x G \rightarrow G \x G)$
where $G$ is a group and the
map $\rc$ satisfies the following requirements
\begin{enumerate}
\item $m \circ \rc = m'$
\item $\rc \circ (m \x \id) = (m \x \id) \circ \rc_{23} \circ \rc_{13}$
\item $\rc \circ (\id \x m) = (\id \x m) \circ \rc_{12} \circ \rc_{13}$
\end{enumerate}
where $m$ is the multiplication and $m'$ is the opposite multiplication in $G$.
In particular, $\rc$ satisfies the set-theoretical
Yang--Baxter equation
\[
\rc_{12} \circ \rc_{13} \circ \rc_{23} = \rc_{23} \circ \rc_{13}
\circ \rc_{12}.
\]
An example of a braided group is the pair $(G^*,\rc)$ where $G^*$ is a
 factorized
group $G$ with multiplication $m(x,y)=x\star y$ and
$\rc$ is given by (\ref{xLR}).

For a set $A$ we will say that category {\it $\cc$ is fibered over $A$} if
\begin{itemize}
\item There is a projection $\pi: \Ob(\cc) \rightarrow A$
\item $\Hom_{\cc}(X,Y) = \emptyset$, if $\pi(X) \ne \pi(Y)$
\end{itemize}

Recall that a monoidal category is a category with a functor
$\otimes : \cc\times \cc \to \cc$. This functor is given together
with the natural transformations $a: {otimes}\circ{\otimes\times
id}\simeq \otimes\circ{id \times \otimes}$ which is called an
associativity constraint. This natural transformation should
satisfy the pentagon identity \cite{ML}. In addition to this, a
monoidal category has an identity object $1\!\!1$ given together with
a system of functorial isomorphisms $l_X: 1\!\!1\otimes X\simeq X$
and $r_X: 1!\!1\otimes X\simeq X$ which satsfy some natural
conditions \cite{ML}.

A monoidal category fibered over a group $G$ is called a {\it
$G$-category} if
\begin{itemize}
\item $\pi(X \otimes Y) = \pi(X)\pi(Y)$
\item $\pi(1\!\!1) = e \in G$
\end{itemize}

Associativity constraint and functorial morphisms $r_X$ and $l_X$
should act fiber-wise. From now on we will work only with strict
monoidal categories: we assume that the associativity constrain is
trivial (see \cite{ML} for details on what this exactly means).

Recall that a monoidal category is called rigid if any object $X$
has a left dual object $X^*$, the injection and evaluation
mappings $i_X: 1\!\!1\to X\otimes X^*$ and $e_X: X^*\otimes X\to
1\!\!1$, and if the triple $X^*, i_X, e_X$ is unique up to an
isomorphism. In a rigid monoidal category double dual is not
necessary isomorphic to the object itself, and, in particular,
each object has left and right duals.

A $G$-category $\cc$ is called {\it rigid} if it is a rigid
monoidal category and in addition to the properties listed above
one has
\[
\pi(X^*) = \pi({}^*X) = \pi(X)^{-1}
\]
where ${}^*X$ and $X^*$ are left and right duals to $X$
respectively. The evaluation and injection morphisms act
fiber-wise.

Now assume that the group $G$ is a braided group.  A category
$\cc$ is called braided rigid $G$-category if it is a rigid
$G$-category and in addition to this it has the following
properties:
\begin{enumerate}

\item There exists a functor $B: \cc \x \cc \rightarrow \cc \x
\cc$ such that the following diagram is commutative
\[
\begin{CD}
\cc \x \cc @>{B}>> \cc \x \cc \\
@V{\pi\x\pi}VV        @VV{\pi\x\pi}V \\
G \x G @>{\check \rc}>> G \x G
\end{CD}
\]
where ${\check \rc} = \rc\circ P$, and $P(x,y) = (y,x)$.  We will
write $B: (X,Y) \rightarrow (X_L(X,Y),X_R(X,Y))$ for the action of
$B$. This property is the lifting of the property 1. of braided
group $G$ to the category $\cc$.

\item The functor $B$ satisfies the following identities ( for
functors $\cc\times \cc\times \cc\to \cc\times \cc\times \cc$)
\[
B\circ(\otimes \times {id})=({id}\times
\otimes)\circ(B\times{id})\circ({id}
\times B)
\]
and the same identity for $B^{-1}$. This properties of $B$ are
liftings of properties 2. and 3. of the braided group $G$ to the
category $\cc$.

\item There exists an isomorphism of functors $c$ which makes the
following diagram commutative
\[
\begin{array}{ccc}
\cc \x \cc &\stackrel{B}{\rightarrow} &\cc \x \cc \\
\otimes \searrow &\stackrel{c}{\Rightarrow} &\swarrow \otimes \\
&\cc
\end{array}
\]
In other words, there exists a system of functorial isomorphisms
\[
c^{X,Y}: X \otimes Y \rightarrow X_L(X,Y) \otimes X_{R}(X,Y).
\]
\item The commutativity constraint should satisfy the hexagon axioms
\[
c^{X\otimes Y,Z} = (c^{X,X_L(Y,Z)} \otimes \id)(\id \otimes c^{Y,Z})
\]
\[
c^{X,Y \otimes Z} = (c^{X,Y} \otimes \id)(c^{X_{R}(X,Y),Z}
\otimes \id)
\]
\end{enumerate}

A braided $G$-category is called a {\it ribbon category} if it in
addition to being a $G$-category has a system of functorial
morphisms $\{\mu_X: X \rightarrow X^{**}\}_{X \in \Ob(\cc)}$ such that
\begin{itemize}
\item $\mu_{X \otimes Y} = \mu_X \otimes \mu_Y$
\item $\mu_{X^*} =(\mu_X^*)^{-1}$
\item $\mu_{1\!\!1}=id$
\end{itemize}
When it will not be misleading we will shorten the name `` rigid
braided ribbon $G$-category'' to ``ribbon $G$-category''.  The
theorem \ref{dg} provides an example of a $G$-category.

When $G_+=e$ are $G_-=G$ the notion of the $G$-category introduced
above is equivalent to the one introduced in \cite{T-2}.

\subsection{The category of $G$-colored diagrams is a ribbon $G$-category}

\begin{thm}\label{dg}
The category of $G$-colored framed diagrams is a ribbon $G^*$-category
where $G^*=G_+\times G_-$.
\end{thm}

\medskip
\noindent
\begin{proof} First, let us check that $\rc$ satisfies the required properties.
The first identity for $\rc$ is equivalent to
\[
x_L(x,y)\cdot x_R(x,y)=x\cdot y
\]
The second identity for $\rc$ is equivalent to
\[
x_L(x, y\cdot z)=x_L(x,y)\cdot x_L(x_R(x,y),z) \,
\]
and
\[
x_R(x, y\cdot z)=x_R(x_R(x,y),z)
\]
and similar identities assure the last property of $\rc$. Here the
multiplication is taken in $G^*$. All these identities are easy to check.

Now let us describe the structure of a $G^*$-category explicitly.
Define the $G^*$ structure on $\dc(G)$ as
\[
\pi((\e_1,x_1),\dots,(\e_n,x_n))=\e_1(x_1)\cdots\e_n(x_n)
\]
Here the product is taken in $G^*$, $x\cdot
y=x_+y_+y_-^{-1}x_-^{-1}$, and $\e(x)$ is defined in
(\ref{invex}).

The monoidal structure is the same as for the category of
diagrams. The tensor product of objects is
\[
\{(\e_1,x_1),\dots,(\e_n,x_n)\}\otimes\{(\s_1,y_1),\dots,(\s_m,y_m)\}=
\{(\e_1,x_1),\dots,(\e_n,x_n),(\s_1,y_1),\dots,(\s_m,y_m)\}
\]
The tensor product of morphisms is shown on Fig.~\ref{t-prod}. The
identity object is the empty set.

\begin{figure}[htbp]
\begin{center}
{\scalebox{0.4}{\includegraphics{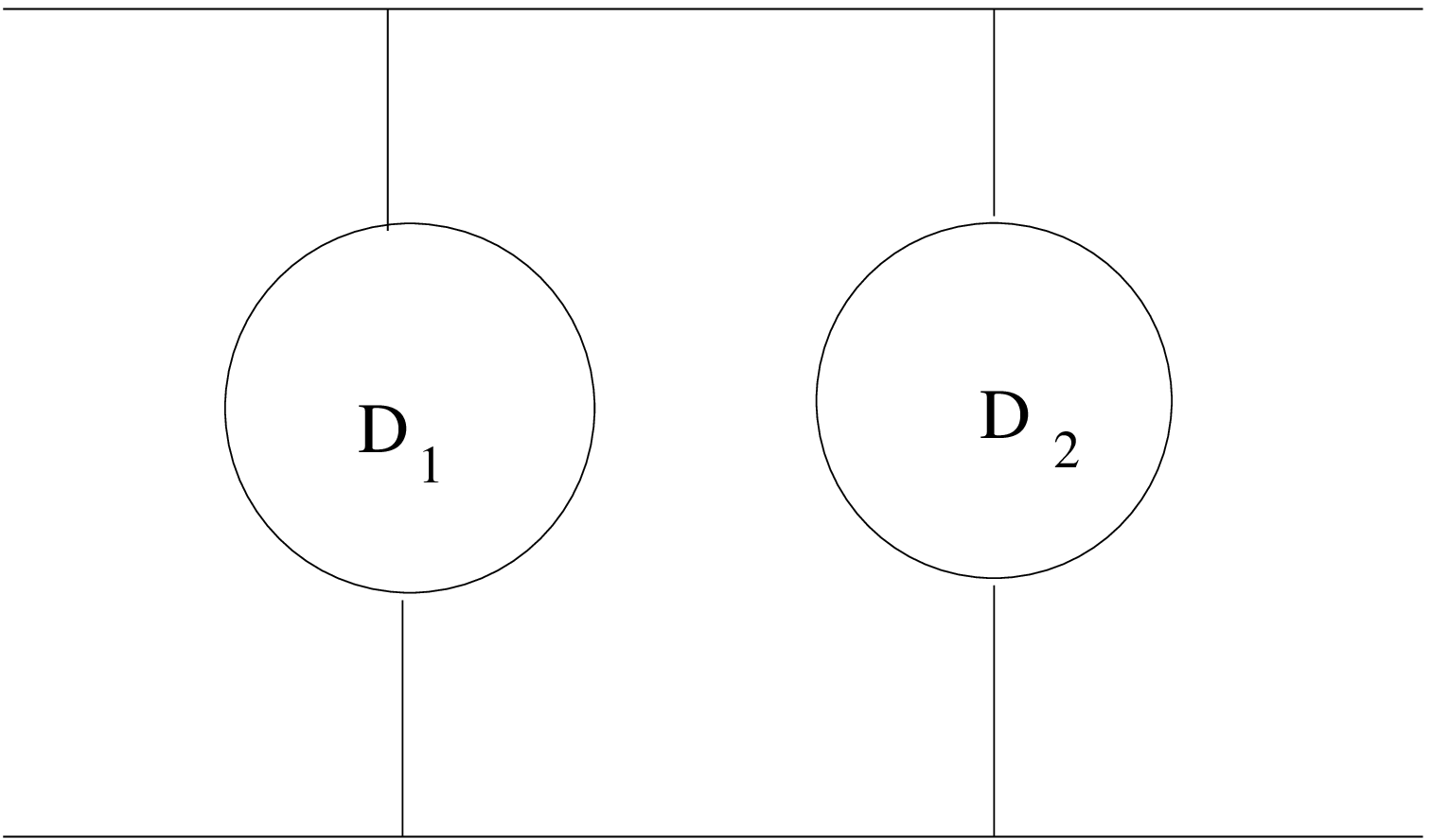}}}
\end{center}
\caption{}
\label{t-prod}
\end{figure}

The object dual to $\{(\e_1,x_1),\dots,(\e_n,x_n)\}$ is
$\{(-\e_n,i(x_n)),\dots,(-\e_1,i(x_1))\}$ with the evaluation and
the injection morphisms given by diagrams from Fig.~\ref{fig-12} and
Fig.~\ref{fig-13} with the $G$-colorings induced by
objects.

\begin{figure}[htbp]
\begin{center}
{\scalebox{0.4}{\includegraphics{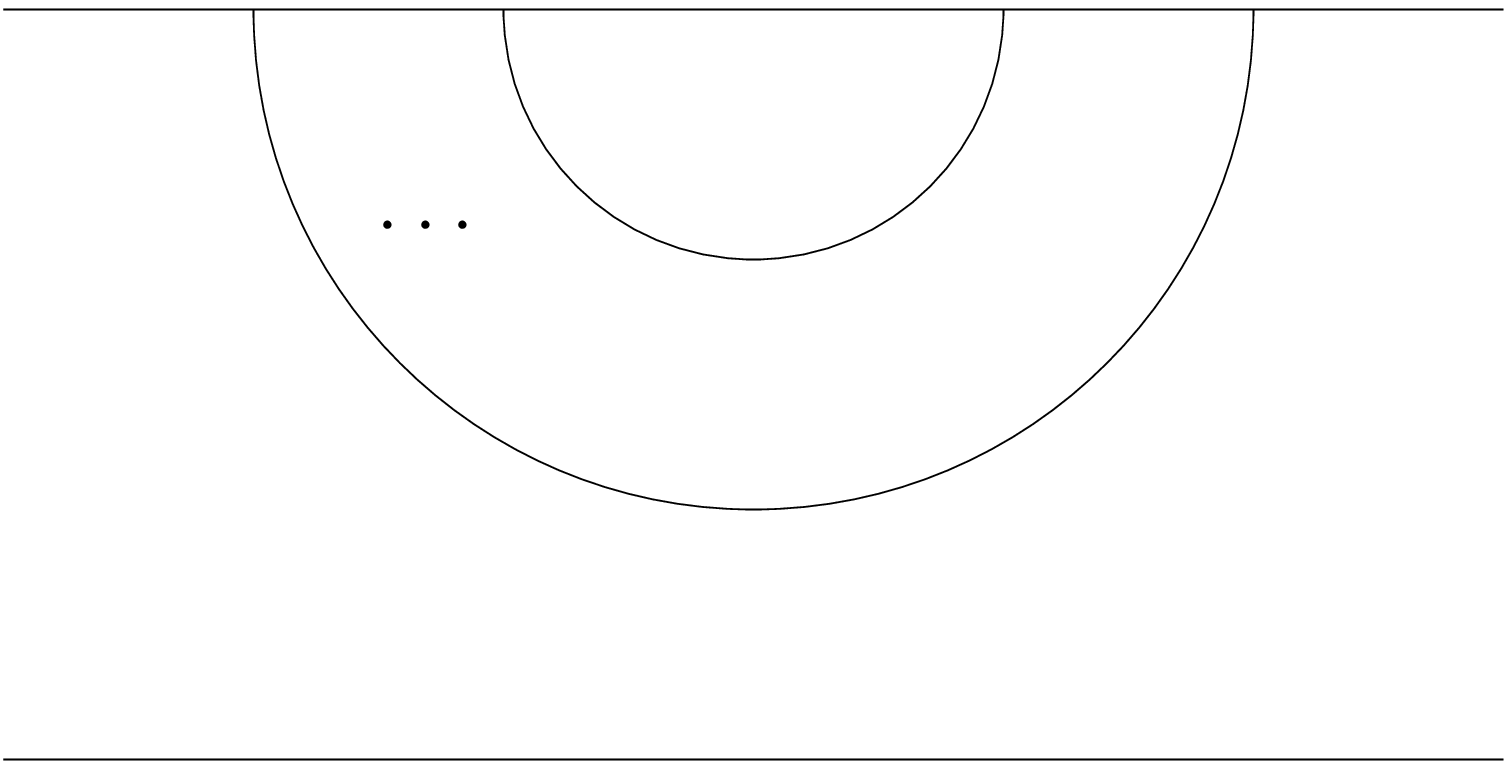}}}
\end{center}
\caption{}
\label{fig-12}
\end{figure}

\begin{figure}[htbp]
\begin{center}
{\scalebox{0.4}{\includegraphics{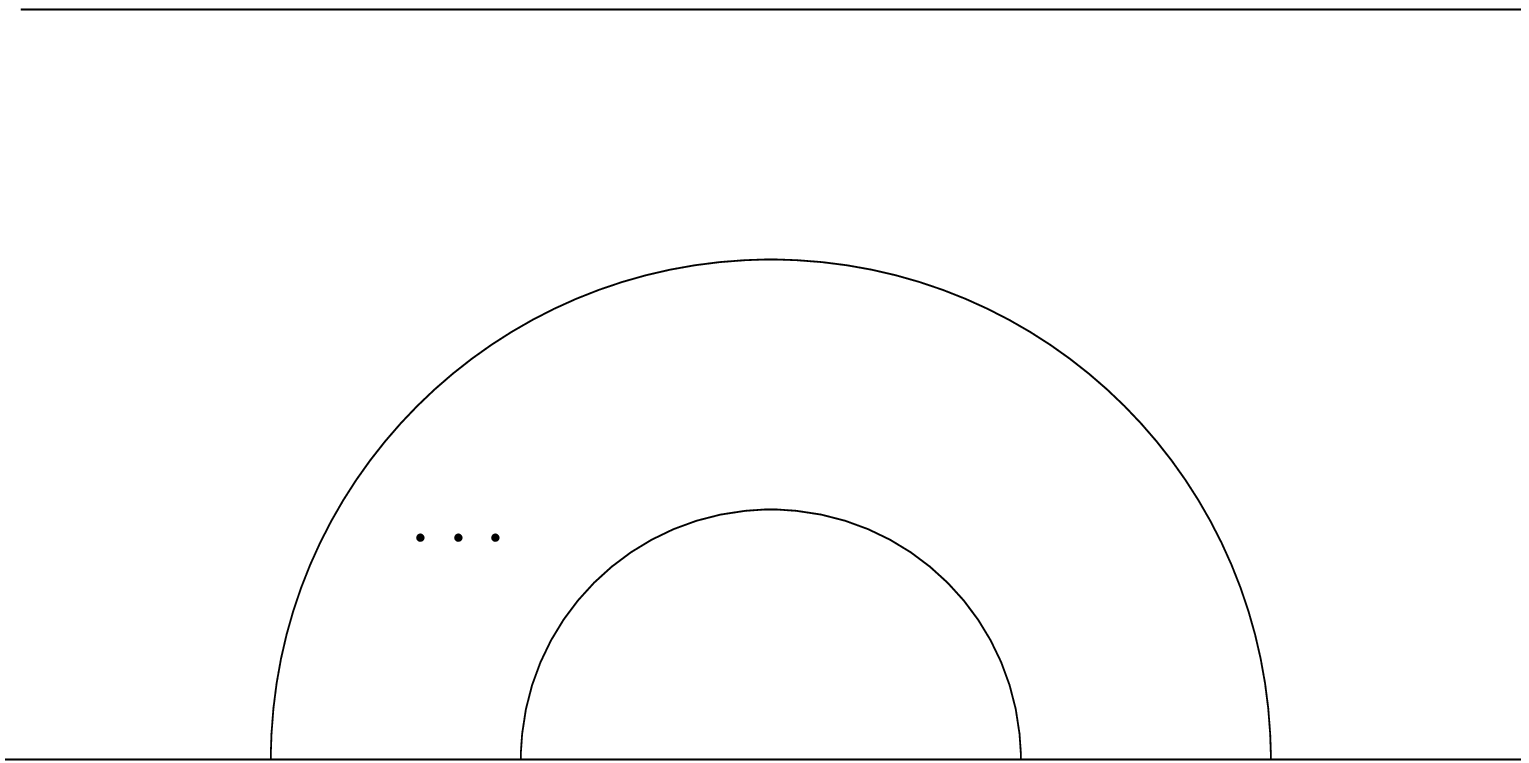}}}
\end{center}
\caption{}
\label{fig-13}
\end{figure}

To describe the braiding, let us first define the functor $B:
\dc(G)\times \dc(G) \to \dc(G)\times \dc (G)$ as follows.
On objects it acts as
\begin{multline*}
B(\{(\e_1,x_1),\dots,(\e_n,x_n)\},\{(\s_1,y_1),\dots,(\s_m,y_m)\})\\
=
(\{(\s_1,y_1^L),\dots,(\s_m,y_m^L)\},\{(\e_1,x_1^{R}),\dots,(\e_n,x_n^{R})\})
\end{multline*}
where $(x_1^L,\dots,x_m^L,y_1^{R},\dots,y_n^{R})$ is the image of
$(x_1,\dots,x_n,y_1,\dots,y_m)$ with respect to the map
\[
(s_n\cdots s_{n+m-1})(s_{n-1}\cdots s_{n+m-2})
\cdots(s_2\cdots s_{n+1})(s_1\cdots s_n): G^{\x(n+m)}
\to G^{\x(n+m)}
\]
where $s_i = {\check \rc}_{ii+1}$.

If $[(D_i,c_i)]$ is a morphism $(\e^{(i)},x^{(i)})\to (\sigma^{(i)},y^{(i)})$,
\[
B((D_1,c_1),(D_2,c_2))=((D_2,c_2'),(D_1,c_1')),
\]
Here colorings $c_1'$ and $c_2'$ are determined by
$c_1$ and $c_2$ and by the corresponding objects.

The commutativity morphism is represented by the diagram on Fig.
~\ref{fig-14} and it is a mapping
\begin{multline*}
\{(\e_1,x_1),\dots,
(\e_n,x_n)\}\otimes\{(\sigma_1,y_1),\dots,(\sigma_m,y_m)\}\\
\to \{(\s_1,y_1^L),\dots,(\s_m,y_m^L)\}\otimes\{(\e_1,x_1^{R}),
\dots,(\e_n,x_n^{R})\}
\end{multline*}
The coloring of the diagram on Fig.~\ref{fig-14} is determined
by the objects.

\begin{figure}[htbp]
\begin{center}
{\scalebox{0.4}{\includegraphics{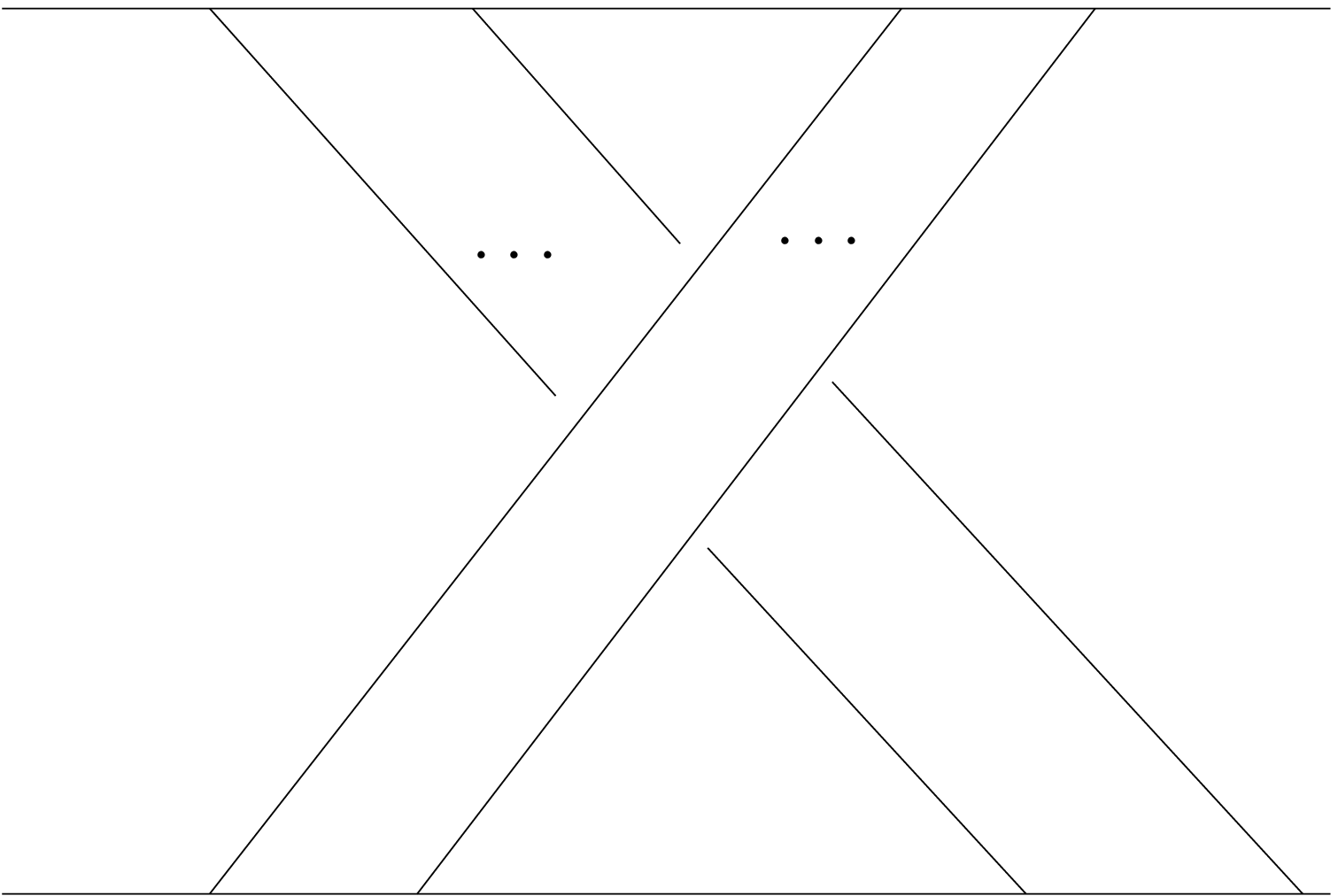}}}
\end{center}
\caption{}
\label{fig-14}
\end{figure}

\end{proof}
\subsection{Elementary diagrams}
The following fact is a key for construction of invariants of
tangles via braided monoidal categories.

\begin{prop}All morphisms in the category ${\mathcal D}(G)$
are compositions of tensor products of elementary diagrams.
Elementary diagrams are given on Fig.~\ref{elem-diagr}.
\end{prop}

\begin{figure}[htbp]
\begin{center}
{\scalebox{0.4}{\includegraphics{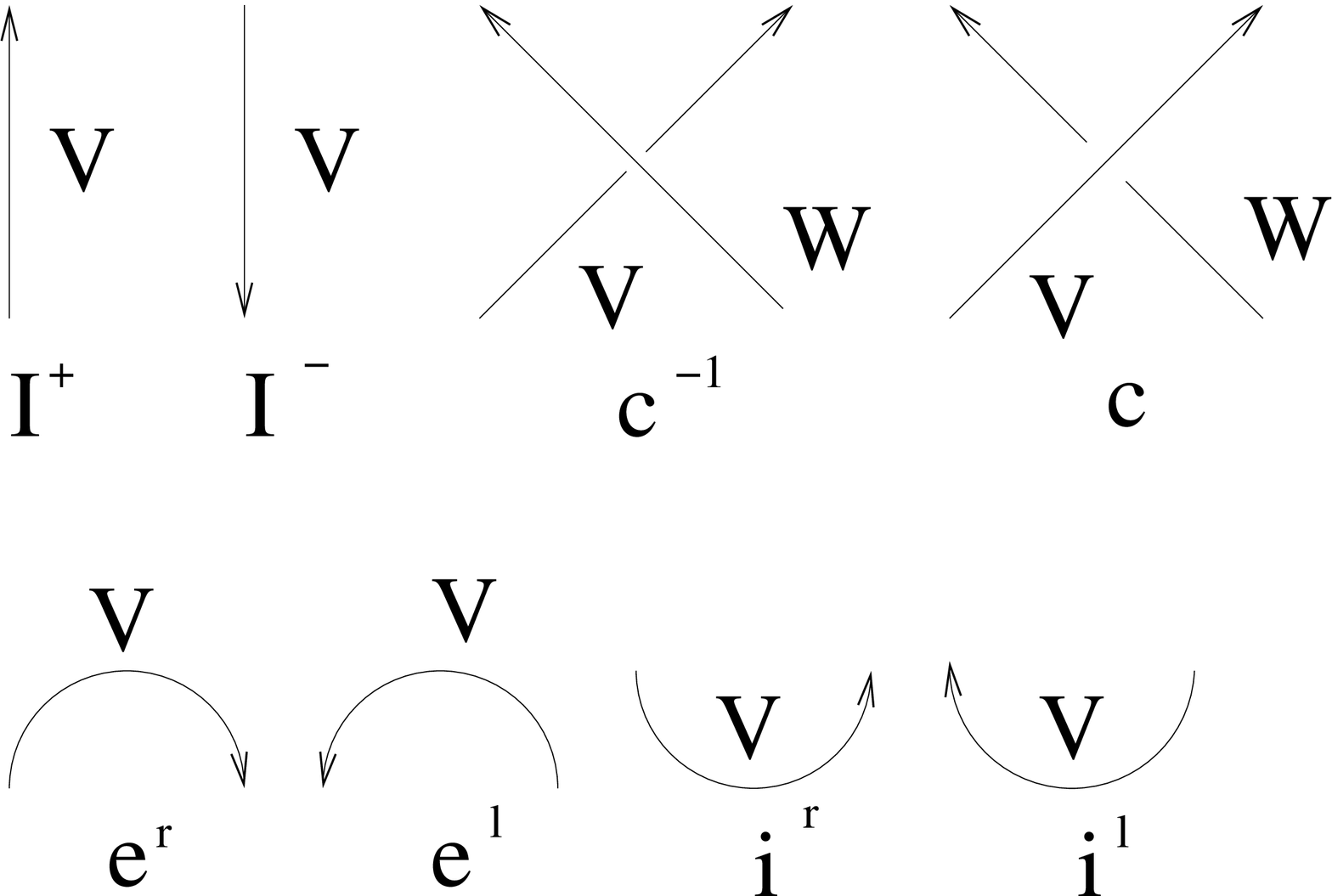}}}
\end{center}
\caption{}
\label{elem-diagr}
\end{figure}

The proof of this proposition and the definition of elementary
diagrams are the same as for the category of framed tangles (see
\cite{T-1}).

\subsection{The category of $\cc$-diagrams}

Let $\cc$ be a braided $G$-category.
\medskip
\noindent
\begin{defin}  A pair $(D,a)$ where $D$ is a diagram and $a\colon
E(D) \rightarrow \Ob(\cc)$ is called a $\cc$-diagram if at each
double point the values of the map $a$ on adjacent edges satisfy
the following conditions (the edges are enumerated as in Fig.~\ref{fig-enum}):
\begin{itemize}
\item If a double point is positive, $X_a=X_L(X_c,X_d), \ X_b=X_R(X_c,X_d)$.
\item If the double point is negative, $X_c=X_L(X_a,X_b), \
X_d=X_R(X_a, X_b)$.
\end{itemize}
\end{defin}
If $G$ is factorizable and $\rc$ is given by (\ref{xLR}), then it is clear
that each $\cc$-diagram defines a $G$-colored diagram defined by the
composition
map $\pi \circ a: E(D) \rightarrow G$.

Now let us define the category $\dc(\cc)$ of $\cc$-diagrams.

{\em Objects} of $\dc(\cc)$ are finite sequences
$\{(\e_1,X_1),\dots,(\e_n,X_n)\}$, where $\e_i = \pm$ and $X_i \in \Ob(\cc)$.

{\em Morphisms} from $\{(\e_1,X_1),\dots,(\e_n,X_n)\}$ to
$\{(\s_1,Y_1),\dots,(\s_m,Y_m)\}$ are framed $\cc$-diagrams $(D,a)$ with
$a(e_i^+) = Y_i^{\e_i}$, $a(e_i^-) = X_i^{\s_i}$.  Here $e_i^+$ are edges of $D$
adjacent to the upper boundary, enumerated from left to right
$i=1,\dots,m$, and $e_i^-$ are edges adjacent to the lower boundary, with
$i=1,\dots,n$.

The identity morphism of $(\e_1,X_1),\dots,(\e_n,X_n)$ is the trivial
braid with the orientation of components defined by $\e_i$ and with $X_i
\in \Ob(\cc)$ assigned to $i$-th strand.

The following theorem is a generalization of Theorem \ref{dg} about
$G$-colored diagrams.

\begin{thm}
The category $\dc(\cc)$ is a ribbon $G$-category.
\end{thm}

The proof is parallel to the one of the theorem \ref{dg}.

\section{Invariants of framed $G$-tangles}\label{sec4}
\subsection{The functor $\Phi\colon\dc(\cc)\to \cc$} As in the case of framed
tangles and tangled framed graphs the first step in
the construction of invariants of
$G$-tangles will be the construction of a rigid $G$-braided
monoidal functor from the category of $\cc$-diagrams to the
category $\cc$.

Let $\cc$ be a ribbon $G$-category.
\begin{thm}
There exists a unique covariant functor $\Phi: \ {\mathcal D}
(\cc)\to \cc$ such that
\begin{itemize}
\item $\Phi(\{(\e_1,X_1), \dots (\e_n,X_n)\})=
X_1^{\e_1} \otimes \dots \otimes X_n^{\e_n}$,
where $X^+=X, \ X^-=X^*$.
\item $\Phi$ is a monoidal functor, i.e.
\[
\Phi((D_1,a_1)\otimes (D_2,a_2))=\Phi((D_1,a_1))\otimes \Phi((D_2,a_2))
\]
where $(D_i,a_i)$ are $\cc$-diagrams.
\item Values of $\Phi$ on elementary diagrams are:
\begin{enumerate}
\item If $I^\e: \  (\e,X)\to (\e,X)$ is the identity morphism
then
\[
\Phi(I^\e)=id_{X^\epsilon}.
\]

\item For morphism  $ e^{r,l}_X:  (\pm,X)\otimes (\mp,X)\to 1$
we have:
\[
\Phi(e^r_X)=e_{X^*}\circ(\mu_X\otimes id): X\otimes X^*\to
X^{**}\otimes X^*\to 1\!\!1  \ ,
\]
\[
\Phi(e^l_X)=e_X: X^*\otimes X\to 1\!\!1 \,
\]

\item For $i^{r,l}: \  1\!\!1\to(\mp,X)\otimes (\pm,X)$ we have
\[
\Phi(i^r_X)= (\mu_X^{-1}\otimes id)\circ i_{X^*}: \ 1\!\!1 \to
X\otimes X^* \,
\]
\[
\Phi(i^l_X)= i_X: \ 1\!\!1\to X^*\otimes X,
\]

\item If $c: \  (+,X)\otimes (+,Y)\to (+,X_L(X,Y))
\otimes (+,X_R(X,Y))$ is the
braiding morphism,
\[
\Phi(c)=c^{X,Y}: \   \ X\otimes Y \to X_L
\otimes X_R \ .
\]
\item For the inverse morphism $c^{-1}: \  (+,X_L)\otimes (+,X_R)\to (+,X)
\otimes (+,Y)$ we have:
\[
\Phi(c^{-1})= (c^{X,Y})^{-1}: \  X_L\otimes
X_R \to X\otimes Y \ .
\]
\end{enumerate}

\item This functor $\Phi$ is rigid monoidal and $G$-braided.
\end{itemize}
\end{thm}

The proof of this theorem is completely parallel to the
corresponding theorem describing invariants of framed tangles.

\subsection{Invariants of framed $G$-tangles}
Let $\cc$ be a ribbon $G$-category and $A=\{A_x\}_{x\in G}$ be a
family of objects such that $\pi(A_x)=x$. In other words $A$ is a
section of  $\pi: \cc\to G$. Assume that objects from this family have the
following property:
\[
B: (A_x, A_y) \to (A_{x_L(x,y)},  A_{x_R(x,y)})
\]
This implies that the braiding morphisms act as:
\[
c^{A_x,A_y}: A_x\otimes A_y\to A_{x_L(x,y)}\otimes  A_{x_R(x,y)}
\]
It also implies that
\[
c^{A^*_x, A_y}: A^*_x\otimes A_y\to A_{x_R(y,x_1)}\otimes A^*_{x_1}
\]
where $x_1\in G$ is such that $x=x_L(y,x_1)$. This can be derived
from the previous formula and from the axioms for the evaluation
and injection morphisms. The action of the braiding on other duals
can be computed similarly.

Now, let us fix such family of objects in $\cc$ to define a
$\cc$-coloring of $D$ for a given $G$-coloring $c$ of the diagram
$D$.
We construct invariants of $G$-tangles as follows:
\begin{itemize}
\item given a $G$-tangle $(t,\rho)$ define the $G$-colored diagram $[(D,c)]$
as in section 2 using the equivalence of categories of $G$-tangles and
$G$-diagrams.
\item given a $G$-colored diagram $[(D,c)]$ define a $\cc$-diagram $[(D,a)]$
as above.
\item apply the functor $\Phi$ to the $\cc$-diagram $[(D,a)]$.
\end{itemize}
It is clear, from the definition of every step here, that the
composition map is an invariant of $(t,\rho)$ with values in
morphisms of the category $\cc$.

In the following section we will describe a $GL_2(\C)$-category associated
with  the quantized universal enveloping algebra of $gl_2(\C)$.

\section{Quantized universal enveloping algebra of $gl_2$}\label{sec5}
\subsection{The algebra $\V$}
The algebra $\V$ over the ring $\C[[h]]$ is generated
by elements $H,G,X$, and $Y$ with defining relations
\[
[H,G]=0, [H,X]=2X,  [H,Y]=-2Y,
\]
\[
[G,X ]=2X, [G,Y]=-2Y,
\]
\[
 [X,Y]=\frac{e^{\frac{hH}{2}}-e^{-\frac{hG}{2}}}{e^{\frac{h}{2}}-e^{-\frac
 {h}{2}}}
\]

The Hopf algebra structure on $\V$ is defined by the action of the
comultiplication on generators:
\begin{eqnarray}
\D H&=&H\otimes 1+1\otimes H, \ \D G=G\otimes 1+ 1\otimes G, \\
\D X&=&X \otimes e^{\frac{hH}{2}}  + 1\otimes X, \ \D Y= Y
\otimes 1 +e^{-\frac{hG}{2}}\otimes Y
\end{eqnarray}

Elements $H, G, X$ and $Y$ "correspond to" elements $2e_{11}, \
2e_{22}, \ e_{12}$ and $e_{21}$ respectively.

The algebra $\V$ is the Drinfeld double of the quantized universal
enveloping algebra $U_h(b)\subset U_h(sl_2)$ where $b$ is a Borel
subalgebra in $sl_2$. As the double of a Hopf algebra it is
quasitriangular with the universal $R$-matrix
\[
R=\exp\left( \frac{h}{4} H\otimes G\right) \prod_{n\geq 0}
(1+e^{\frac{h}{2}}(e^{\frac{h}{2}}-e^{-\frac {h}{2}})^2X\otimes Y
e^{-nh})
\]
This is the element of $\V^{\otimes 2}$ which one should consider
as a formal power series in $h$.

Since $R$ is the universal $R$-matrix it satisfies the following
identities:
\begin{equation}\label{br}
R \D(a)R^{-1}= \sigma\cdot\D(a)
\end{equation}
\[
(\D \oti id)(R)= R_{13} R_{23}
\]
\[
(id \oti \D)(R) =R_{13} R_{12}
\]
where $\sigma$ is the permutation operator $\sigma(a\otimes
b)=b\otimes a$. In particular, $R$ satisfies the Yang-Baxter equation
\[
R_{12}R_{13}R_{23}=R_{23}R_{13} R_{12}
\]

\subsection{The inner automorphism $\rc$}
Define the inner automorphism \\ $\rc: \V^{\oti 2}[[h]]\to\V^{\oti
2}[[h]]$ as
\begin{equation}\label{rc}
\rc(x\oti y)=R(x\oti y) R^{-1}
\end{equation}

It is easy to compute the action of $\rc$ on generators.

\begin{thm} The following identities hold.
\[
\rc(1\oti e^{\frac{hH}{2}})=(1\oti
e^{\frac{hH}{2}})(1+e^{\frac{h}{2}}(e^{\frac{h}{2}}-e^{-\frac
{h}{2}})^2e^{-\frac{hH}{2}}X\otimes Ye^{\frac{hG}{2}})^{-1}
\]
\[
\rc(1\oti e^{\frac{hG}{2}})=(1\oti
e^{\frac{hG}{2}})(1+e^{\frac{h}{2}}(e^{\frac{h}{2}}-e^{-\frac
{h}{2}})^2e^{-\frac{hH}{2}}X\otimes Ye^{\frac{hG}{2}})^{-1}
\]
\[
\rc(X\oti 1)=X\oti e^{\frac{hG}{2}}
\]
\[
\rc(1\oti Y)=e^{-\frac{hH}{2}}\oti Y
\]
\end{thm}
The theorem follows immediately from the commutation relations
between generators and from the equation
\[
f(zq^{-1};q)=(1+zq^{-1})f(z;q)
\]
for the function
$$
            f(z;q)=\prod_{n=0}^\infty (1+zq^n).
 $$   

The action of $\rc$ on elements $1\oti X$, $Y\oti 1$,
$e^{\frac{hH}{2}}\oti 1$, and $e^{\frac{hG}{2}}\oti 1$ can be
derived from the formulae above and from the identity (\ref{br}).

The Yang-Baxter equation for $R$ implies the Yang--Baxter equation
for $\rc$:
\[
\rc_{12}\cdot\rc_{13}\cdot\rc_{23}=\rc_{23}\cdot\rc_{13}\cdot\rc_{12}
\]

\section{The algebra $\U$}\label{sec6}

The algebra $\U$ is generated over $\C[t,t^{-1}]$ by
elements $K,L,E$ and $F$ with the following defining relations
\[
KL=LK, \ KE=t^2EK, \ KF=t^{-2}FK,
\]
\[
LE=t^2EL, \ LF=t^{-2}FL,
\]
\[
EF-FE=(t-t^{-1})(K- L^{-1})
\]
The center of $\U$ is generated freely by Laurent polynomials in
$KL^{-1}$ and
\begin{equation}\label{c}
c=EF+Kt^{-1}+L^{-1}t
\end{equation}

This is a Hopf algebra with
\[
\D(K)=K\oti K, \ \D(L)=L\oti L,
\]
\[
\D(E)=E\oti K+1\oti E, \ \D(F)=F\oti 1+L^{-1}\oti F \ .
\]

The map $\phi: \U\to \V$ acting on generators as
\[
\phi(K)=\exp(\frac{hH}{2}), \ \phi(L)=\exp(\frac{hG}{2}),
\phi(t)=e^{\frac{h}{2}}
\]
\[
\phi(E)=(e^{\frac{h}{2}}-e^{-\frac {h}{2}})X, \
\phi(F)=(e^{\frac{h}{2}}-e^{-\frac {h}{2}})Y
\]
extends to a homomorphism of Hopf algebras.

The algebra $\U$ is not quasitriangular. Instead, there
is an outer automorphism of the division ring $\bar{\U^{\oti 2}}$
of $\U^{\oti 2}$ which we denote by the same letter $\rc$ as the
automorphism (\ref{rc}) which acts on generators as
\[
\rc(1\oti K)=(1\oti K)(1+tK^{-1}E\oti FL)^{-1}
\]
\[
\rc(1\oti L)=(1\oti L)(1+tK^{-1}E\oti FL)^{-1}
\]
\[
\rc(E\oti 1)=E\oti L
\]
\[
\rc(1\oti F)=K^{-1}\oti F
\]
Define its action on generators $K\oti 1$, $L\oti 1$, $1\oti E$,
and $F\oti 1$ such that
\[
\rc(\D (a))=\sigma\circ \D(a)
\]
where $a$ is one of the generators of $\U$.

It is clear that the homomorphism $\phi$ brings the outer
automorphism $\rc$ to (\ref{rc}).

\section{The algebra $\ue$}\label{sec7}

Let $\e$ be a primitive root of $1$ of an odd degree $\ell$.
Denote by $\ue$ the specialization of $\U$ to $t=\e$. The
following theorem is a version of the corresponding facts for
simple Lie algebras proved in \cite{DC-K}.

\begin{thm}
\begin{itemize}
\item Elements $E^\ell$, \ $F^\ell$, \ $K^{\pm\ell}$, and
$L^{\pm\ell}$ generate a central
subalgebra $Z_0\subset \ue$.

\item $Z_0$ is a Hopf subalgebra with
\[
\Delta(K^\ell)=K^\ell\oti K^\ell, \ \Delta(L^\ell)=L^\ell\oti
L^\ell,
\]
\[
\Delta(E^\ell)=E^\ell\oti K^\ell+1\oti E^\ell, \
\Delta(F^\ell)=F^\ell\oti 1+L^{-\ell}\oti F^\ell .
\]

\item The algebra $\ue$ is a free $Z_0$-module of dimension
$\ell^4$.

\item The center $Z(\ue)$ is generated by $Z_0$ and by the element
(\ref{c}) modulo the relation
\[
\prod_{j=0}^{\ell-1}(c-K\e^{j+1}-L^{-1}\e^{-j-1})=E^\ell F^\ell
\]

\item Let $\alpha, \beta, a$ and $b$ be coordinates on the group
$B_+\times B_-$ such that for $b_\pm\in B_\pm$ we have:
\[
b_+=\left(\begin{array} {cc}1& \beta \\ 0& \alpha\end{array}\right) \ , \
b_-=\left(\begin{array} {cc}a& 0 \\ b& 1\end{array}\right)
\]
Then the map $F^\ell\to ba^{-1}$, $E^\ell\to \beta$,
$K^\ell\to \alpha$, $L^\ell\to a$ is a homomorphism of Hopf
algebras $Z_0\to C(B_+\times B_-)$

\item $\ue$ is semisimple over a Zariski open subvariety of
$Spec(Z_0)\simeq B_+\times B_-$.

\end{itemize}
\end{thm}

Let $x\in GL_2^*$ be an irreducible
$Z_0$-character and $I_x\subset \ue$ be the corresponding ideal.
The quotient algebra
\[
A_x=\ue/I_x
\]
is finite-dimensional of dimension $\ell^4$. There are three
natural structures of a left module on $A_x$. For $a\in \ue$
denote by $[a]$ the class of $a$ in $A_x$. Then these three
actions are:
\begin{itemize}
\item $\pi(a)[b]=[ab]$, \item $\phi(a)[b]=[bS(a)]$, \item
$\psi(a)[b]=[bS^{-1}(a)]$.
\end{itemize}

Assume that $x\in GL_2^*$ is generic, i.e. that
$A_x$ is semisimple. Fix an isomorphism of algebras $\phi_x: A_x
\simeq \oplus_{i=1}^n Mat(k_i)$. For the algebra $\ue$ it is known
\cite{DC-K} that $n=\ell^2$ and $k_i=\ell$ for all $i=1,\dots, n$.
Define
\[
t: A_x\to \C, \ t(a)=\sum_{i=1}^n t_i Tr(\phi^i_x(a))
\]
where $Tr$ is the matrix trace in $Mat(k_i)$ and $\phi^i_x:
A_x\to Mat(k_i)$ is the $i$-th component of $\phi_x$.
It is clear that $t(a)$ does not depend on a particular choice of
$\phi_x$. Indeed, any other such isomorphism differs from $\phi_x$
by an inner automorphism of $\oplus_{i=1}^n Mat(k_i)$. Since
trace is cyclically invariant, the value of $t(a)$ for such
isomorphism will be the same as for $\phi_x$.
Thus, for generic $x$, we have an invariant bilinear form on
$A_x$:
\[
(a,b)=t(ab).
\]
It is a scalar product if $t_i\neq 0$ for each $i=1,\dots, n$.

Fix a scalar product on $A_x$ as above. This
gives an isomorphism of vector spaces $A_x^*\simeq A_x$.
It is easy to verify that the pairing between $\ue$-modules
$(A_x, \phi)$ and $(A_x, \pi)$ given by the map
\begin{equation}\label{e}
e_x: (A_x, \phi)\otimes (A_x, \pi)\to \C
\end{equation}
acting as $ a\otimes b \mapsto t(ab)$ is $\ue$-invariant with
respect to the diagonal action. Indeed, using Sweedler's notation
$\Delta(c)=\sum_c c^{(1)}\otimes c^{(2)}$ for the comultiplication
of element $c$, we have:
\[
\e_x(\sum_c aS(c^{(1)}) \otimes c^{(2)}b)=\sum_c
t(aS(c^{(1)})c^{(2)}b)=\e(c)t(ab)
\]
Similarly
\[
e_x(\sum_c c^{(1)}a\otimes bS^{-1}(c^{(2)}))=\sum_c
t(c^{(1)}abS^{-1}(c^{(2)}))=\e(c)t(ab)
\]
Therefore the map $e_x: (A_x, \pi)\otimes
(A_x, \psi)\to \C$ defined as in (\ref{e}) is also $\ue$-invariant.

Let linear mapping $i_x: \C\to A_x\otimes A_x$ be defined by the formula
\[
i_x(1)\mapsto\sum_i e_i\otimes e^i,
\]
$\{e_i\}$ is a linear basis of $A_x$ and $\{ e^i\}$ the corresponding
dual basis.
It is easy to see that it is a morphism of $\ue$-modules
$\C\to (A_x,\pi)\otimes (A_x,\phi)$ and $\C\to (A_x,\psi)\oti
(A_x,\pi)$. Indeed, let $a_i^j=t(ae_ie^j)$ for any $a\in A_x$, then
\[
\sum_a\sum_ia^{(1)}e_i\oti e^iS(a^{(2)})=\sum_a\sum_{i,j}
(a^{(1)})_i^je_j\otimes e^i(S(a^{(2)}))=\sum_a\sum_{j}
e_j\otimes e^ja^{(1)}S(a^{(2)})=\e(a)\sum_i e_i\oti e^i
\]
which implies the first statement and the second statement
can be proved similarly.

Thus, for the object $(A_x,\pi)$ we have the left dual
$(A_x,\phi)$ (and the right dual $(A_x,\psi)$).

\begin{thm}The subspace $Z_0\oti Z_0\subset \ue\otimes
\ue$ is invariant with respect to the action of the automorphism $\rc$.
\end{thm}
\begin{proof}
From the action  of $\rc$ on generators of $\ue$  and from
the relations between generators we have:
\[
\rc(1\otimes K^\ell)=(1\otimes K^\ell)(1+K^{-\ell}E^\ell\otimes
F^\ell L^\ell)^{-1}  \,
\]
\[
\rc(1\otimes L^\ell)=(1\otimes L^\ell)(1+K^{-\ell}E^\ell\otimes
F^\ell L^\ell)^{-1}  \,
\]
\[
\rc(E^\ell\otimes 1)=E^\ell\otimes L^\ell  \,
\]
\[
\rc(1\otimes F^\ell )=K^{-\ell}\otimes F^{\ell}  \ .
\]
The comultiplication acts on $\ell$-th powers of generators as:
\[
\D(K^\ell)=K^\ell\oti K^\ell, \ \ \D(L^\ell)=L^\ell\oti L^\ell \,
\]
\[
\D(E^\ell)=E^\ell\oti K^\ell+1\oti E^\ell \,
\]
\[
\D(F^\ell)=F^\ell\oti 1+L^{-\ell}\oti F^\ell \ .
\]
These formulae and the defining property $\rc(\D(a))=\sigma\circ
\D(a)$ describe the action of $\rc$ on generators of $Z_0\oti Z_0$.
In particular, it is clear that the image is in $Z_0\oti Z_0$.
\end{proof}

Comparing the action of $\rc$ on generators of $Z_0\oti Z_0$ with
the identification of $Z_0$ and $C(GL_2^*)$ we have the following
statement.
\begin{thm}The automorphism $\rc$ is the pull-back of the mapping
$b:GL_2^*\times GL_2^*\to GL_2^*\times GL_2^*$ acting as follows.
First, identify $GL_2^*$ with the Zariski open subvariety in $GL_2$
via the factorization mapping. Then the mapping $b$ acts as
\[
(x,y)\mapsto (x_L(x,y), x_R(x,y))
\]
where $x_L(x,y)=x_-yx_-^{-1}$ and
$x_R(x,y)=(x_L(x,y))_+^{-1}x(x_L(x,y))_+$.
\end{thm}

Let $I_x$ be the ideal in $\ue$ with the $Z_0$ character $x\in
GL_2^*$. From now on we will consider only generic points $x$ and
therefore we can identify $GL_2$ with $GL_2^*$. The two theorems
have an important corollary. The mapping $\rc$ acts as:
\[
\rc(I_x\oti I_y)\subset I_{x_R(x,y)}\oti I_{x_L(x,y)}
\]
This implies that the mapping $\rc$ induces the isomorphism of
algebras:
\[
\rc(x,y): A_x\oti A_y\to A_{x_R(x,y)}\oti A_{x_L(x,y)} \,
\]
This mapping is also an isomorphism of the tensor product of
left $\ue$-modules $(A_x, \pi)$.

Finally, consider the category of $\ue$-modules generated by
tensor products of $(A_x,\pi)$ and their duals. It is clear that
this category is a braided rigid monomial $G$-category with
$G=GL_2^*$ and with the braiding given by the composition
$\sigma\circ \rc$.

\section{Conclusion}

In this paper we have constructed invariants of tangles with flat
connections in their complements. An example of such construction
is described for $GL_2(\C)$. This is rather simple example
related to quantum invariant constructed in \cite{Ka}. More
interesting examples related to irreducible representations of
quantized universal enveloping algebras of simple Lie algebras
will be analyzed in a separate paper.

The complement of a tangle is a rather special 3-manifold. The
construction of invariants of 3-manifolds with $G$-flat
connections in them for any simple Lie group $G$ is the next step.
The case of $G=PSL_2(\C)$ was studied in \cite{BB}.

An interesting, but somewhat speculative question: how to
relate the constructed invariants to a topological quantum field theory
defined ``phenomenologically'' in terms of functional
integrals. We expect that this theory will be Chern--Simons
theory with complex simple $G$. The corresponding boundary
conformal field theory should be complex Wess--Zumino--Witten
theory on a surface with boundary with boundary operators
"parametrized" by elements of $G$.

\end{document}